\documentclass{article}
\usepackage{amssymb,amsmath}
\usepackage{graphicx,enumerate}

\title{Reticulation of a quantale, pure elements and new transfer properties}

\author{George Georgescu \\ \footnotesize University of Bucharest\\ \footnotesize Faculty of Mathematics and Computer Science\\ \footnotesize Bucharest, Romania\\ \footnotesize Email: georgescu.capreni@yahoo.com}

\date{}

\begin{document}
\maketitle

\begin{abstract}
We know from a previous paper that the reticulation of a coherent quantale $A$ is a bounded distributive lattice $L(A)$ whose prime spectrum is homeomorphic to $m$ - prime spectrum of $A$. In this paper we shall prove several results on the pure elements of the quantale $A$ by means of the reticulation $L(A)$. We shall investigate how the properties of $\sigma$ - ideals of $L(A)$ can be transferred to pure elements of $A$. Then the pure elements of $A$ are used to obtain new properties and characterization theorems for some important classes of quantales: normal quantales, $mp$ - quantales, $PF$ - quantales, purified quantales and $PP$ - quantales.

\end{abstract}

\newtheorem{definitie}{Definition}[section]
\newtheorem{propozitie}[definitie]{Proposition}
\newtheorem{remarca}[definitie]{Remark}
\newtheorem{exemplu}[definitie]{Example}
\newtheorem{intrebare}[definitie]{Open question}
\newtheorem{lema}[definitie]{Lemma}
\newtheorem{teorema}[definitie]{Theorem}
\newtheorem{corolar}[definitie]{Corollary}

\newenvironment{proof}{\noindent\textbf{Proof.}}{\hfill\rule{2mm}{2mm}\vspace*{5mm}}

\section{Introduction}

 \hspace{0.5cm} The quantales \cite{Rosenthal},\cite{Eklud} and the frames \cite{Johnstone} are structures that generalise the lattices of ideals, filters and congruences for various classes of algebras. Several algebraic and topological properties of rings, distributive lattices, $l$-groups and $l$-rings, $MV$-algebras,$BL$-algebras, residuated lattices, etc. can be extended to quantales and frames. To work in this abstract setting is not only a unification of some existing particular results, but also an efficient way to prove new properties for many types of algebras.

 The pure elements in a quantale were introduced in \cite{PasekaRN}. They constitute an abstractization of pure ideals of rings \cite{Borceux1} (named virginal ideals in \cite{Borceux} and neat ideals in \cite{Johnstone}), $\sigma$ - ideals of bounded distributive lattices \cite{Cornish},\cite{Cornish1},\cite{GeorgescuVoiculescu}, pure ideals in $MV$-algebras \cite {Cavaccini}, pure filters in $BL$-algebras \cite{g} and residuated lattices \cite{Muresan}, etc. Similar notions of pure elements can be found in other kind of multiplicative lattices: multiplicative ideal structures \cite{GeorgescuVoiculescu2} and $2$-side carriers \cite{SimmonsC}.

 On the other hand, the reticulation of a coherent quantale $A$ is a bounded distributive lattice $L(A)$ whose $m$ - prime spectrum $Spec(A)$ is homeomorphic to the prime spectrum $Spec_{Id}(L(A))$ of $L(A)$ (see \cite{Georgescu}). Then $L(A)$ unifies various notions of reticulations defined for commutative rings \cite{Simmons},\cite{Johnstone}, l-rings \cite{Johnstone}, $MV$-algebras \cite{Belluce}, $BL$ -algebras \cite{g}, residuated lattices \cite{Muresan}, etc.

 The reticulation is a functorial construction: from a category of algebras to the category of bounded distributive lattices. The reticulation functor allows us a transfer of properties from bounded distributive lattices to algebras and vice-versa.

 This paper studies the pure elements of a coherent quantale $A$ by means of the reticulation $L(A)$. We investigate how some results on the $\sigma$ - ideals of the lattice $L(A)$ can be transferred to the pure elements of $A$. We use the properties of pure elements in order to obtain new characterization theorems for some important classes of quantales: normal quantales, $mp$ - quantales, $PF$ - quantales, purified quantales and $PP$ - quantales.

 Now we shall describe the content of the paper. In Section 2 we recall from \cite{Rosenthal},\cite{Eklud} some basic notions and results in quantale theory: $m$ - prime elements, radical elements, $m$ - prime and maximal spectra with spectral and flat topologies, regular and max - regular elements. Section 3 contains some fundamental facts on the reticulation $L(A)$ of a coherent quantale $A$: axiomatic definition, arithmetic, algebraic and topological constructions, the isomorphism between the Boolean centers $B(A)$ and $B(L(A))$ of $A$ and $L(A)$, the preservation theorems for annihilators,etc.

 Section 4 concerns the pure elements in a coherent quantale $A$. We know from \cite{GeorgescuVoiculescu2} that the set $Vir(A)$ of pure elements of $A$ is a spatial frame. In \cite{GeorgescuVoiculescu2}, a lot of properties of $Vir(A)$ were established in the more large framework of multiplicative ideals structures (= $mi$ - structures). We continue the line of \cite{GeorgescuVoiculescu2} and obtain new results on the pure elements. We define the weakly pure elements (= $w$ - pure elements), a notion that enlarges the class of pure elements. Mainly we study the relationship between the operators $Vir(\cdot )$, $Ker(\cdot )$ and $O(\cdot )$ (introduced in \cite{GeorgescuVoiculescu2}) and the pure and $w$ - pure elements of $A$. The main results of the section establish how the reticulation commutes with $Vir(\cdot )$, $Ker(\cdot )$ and $O(\cdot )$. In this way one obtains the relationship between the pure and $w$ - pure elements of $A$ and the $\sigma$ - ideals of the lattice $L(A)$.

In Section 5 we continue the study of pure elements. Firstly, we prove that there exists a surjective continuos function from the prime spectrum of the frame $Vir(A)$ to the Pierce spectrum $Sp(A)$ of $A$. Secondly, in the continuation of \cite{GeorgescuVoiculescu2} we investigate the pure elements in a normal quantale. Various properties that characterize the normal coherent quantales point out the role of pure elements and operators $Vir(\cdot )$, $Ker(\cdot )$ and $O(\cdot )$ in studying this class of quantales (see Proposition 5.6).

Section 6 contains a lot of results on pure elements in $PF$ - quantales, a quantale abstractization of $PF$ - rings. If $A$ is $PF$ - quantale, then we characterize its pure elements as intersections $\bigwedge(Min(A)\bigcap E)$, where $E$ is a closed subset of the minimal $m$ - prime spectrum $Min(A)$ of $A$ (endowed with the restriction of spectral topology on $Spec(A)$)). We prove the equality of $Min(A)$ with the maximal spectrum $Max(Vir(A)$ of the frame $Vir(A)$. Another theorem of the section shows that for any $PF$ - quantale $A$, the frame $Vir(A)$ is hyperarchimedean.

In Section 7 we define the purified quantales as a generalization of purified rings, introduced in \cite{Aghajani}. The main result of the section is a characterization theorem of the purified quantales.

The $PP$ - quantales, introduced in Section 8, generalize the $PP$ - rings ( = Baer rings). We prove that a semiprime quantale $A$ is a $PP$ - quantale if and only if the reticulation $L(A)$ is a Stone lattice, extending a theorem of Simmons from \cite{Simmons}. By using this result one obtains some characterization theorems of $PP$ - quantales.

We mention that most of the results of this paper extend some theorems that appear in the case of commutative rings (\cite{Aghajani},\cite{Al-Ezeh},\cite{Al-Ezeh2},\cite{Artico},\cite{Borceux1},\cite{Borceux},\cite{Johnstone},\cite{c},\cite{Simmons},\cite{Tar2},\cite{Tar3}), bounded distributive lattices (\cite{Al-Ezeh1},\cite{Cornish},\cite{Cornish1},\cite{GeorgescuVoiculescu},\cite{Johnstone},\cite{Pawar},\cite{Simmons},\cite{Speed}), MV-algebras (\cite{Belluce},\cite{Cavaccini}), $BL$ -algebras \cite{g}, residuated lattices (\cite{Muresan},\cite{Rasouli}),etc.

\section{Preliminaries on quantales}

 \hspace{0.5cm} This section contains some basic notions and results in quantale theory \cite{Rosenthal}, \cite{Eklud}. Recall from \cite{Rosenthal}, \cite{Eklud} that a {\emph{quantale}} is an algebraic structure
 $(A,\lor, \land, \cdot, 0, 1  )$ such that $(A,\lor, \land, 0, 1  )$ is a complete lattice and $(A, \cdot )$ is a semigroup with the property that the multiplication $\cdot$ satisfies the infinite distributive laws: for all $a\in A$ and $X\subseteq A$, we have $a\cdot \bigvee X = \bigvee\{a\cdot x|x\in X\}$ and
 $(\bigvee X)\cdot a = \bigvee\{x\cdot a|x\in X\}$.

 Let $(A,\lor, \land, \cdot, 0, 1  )$ be a quantale and $K(A)$ the set of its compact elements. $A$ is said to be {\emph{integral}} if $(A, \cdot, 1)$ is a monoid and {\emph{commutative}}, if the multiplication $\cdot$ is commutative. A {\emph{frame}} is a quantale in which the multiplication coincides with the meet \cite{Johnstone}. The quantale $A$ is {\emph{algebraic}} if any $a \in A$ has the form $a= \bigvee  X$ for some subset $X$ of $K(A)$. An algebraic quantale $A$ is {\emph{coherent}} if $1 \in K(A)$ and $K(A)$ is closed under the multiplication.
Throughout this paper, the quantales are assumed to be integral and commutative. Often we shall write $ab$ instead of $a \cdot b$. We fix a quantale $A$.

\begin{lema}
\cite{Birkhoff} For all elements $a, b, c$ of the quantale $A$ the following hold:
\newcounter{nr}
\begin{list}{(\arabic{nr})}{\usecounter{nr}}
\item If $a\lor b =1$ then $a\cdot b= a \land b $;
\item If $a\lor b =1$ then $a^{n}\lor b^{n}=1$ for all integer number $n\geq 1$;
\item If $a\lor b=a\lor c=1$ then $a\lor (b\cdot c)= a \lor (b\land c)=1$;

\end{list}
\end{lema}

On each quantale $A$ one can define a residuation operation $a \rightarrow b=\bigvee \{ x|a x \leq b\}$ and a negation operation $a^{\bot } =a \rightarrow 0 =\bigvee \{x |ax=0\}$. Thus for all $a, b, c \in A$ the following equivalence holds: $a\leq b\rightarrow c$ if and only if $ab \leq c$, so $(A, \lor, \land, \cdot, \rightarrow, 0, 1)$ becomes a (commutative) residuated lattice \cite{Galatos}.

In this paper we shall use without mention the basic arithmetical properties of a residuated lattice \cite{Galatos}.

An element $p <1$ of $A$ is $m$-{\emph{prime}} if for all $a, b \in A$, $ab \leq p$ implies $a \leq b$ or $b \leq p$. If $A$ is an algebraic quantale, then $p<1$ is $m$-prime if and only if for all $c, d \in K(A)$, $cd \leq p$ implies $c \leq p$ or $d \leq p$. Let us introduce the following notations: $Spec(A)$ is the set of $m$-prime elements and $Max(A)$ is the set of maximal elements of $A$. If $1 \in K(A)$ then for any $a <1$ there exists $m \in Mar(A)$ such that
$a \leq m$. The same hypothesis $1 \in K(A)$ implies that $Max(A) \subseteq Spec(A)$.

 The main example of quantale is the set $Id(R)$  of ideals of a (unital) commutative ring $R$ and the main example of frame is the set $Id(L)$  of ideals of a bounded distributive lattice $L$. Thus the set $Spec(R)$ of prime ideals in $R$ is the prime spectrum of the quantale $Id(R)$ and the set of prime ideals in $L$ is the prime spectrum of the frame $Id(L)$.

Following \cite{Rosenthal}, the {\emph{radical}} $\rho(a)=\rho_A(a)$ of an element $a \in A$ is defined by $\rho_A(a)=\bigwedge \{p\in Spec(A)|a \leq p\}$; if $a=\rho(a)$ then $a$ is a radical element. We shall denote by $R(A)$ the set of radical elements of $A$. The quantale $A$ is said to be {\emph{semiprime}} if $\rho(0)=0$.

\begin{lema}
\cite{Rosenthal} For all elements $a, b \in A$ the following hold:
\usecounter{nr}
\begin{list}{(\arabic{nr})}{\usecounter{nr}}
\item $a \leq \rho(a)$;
\item $\rho(a \land b)=\rho (ab)=\rho(a) \land \rho(b)$;
\item $\rho(a)=1$ iff $a=1$;
\item $\rho(a \lor b)=\rho(\rho(a) \lor \rho(b))$;
\item $\rho(\rho(a))=\rho(a)$;
\item $\rho(a) \lor \rho(b)=1$ iff $a \lor b=1$;
\item $\rho(a^n)=\rho(a)$, for all integer $n \geq 1$.
\end{list}
\end{lema}

For an arbitrary family $(a_i)_{i\in I} \subseteq A$, the following equality holds: $\rho(\displaystyle \bigvee_{i \in I}a_i)=\rho(\bigvee_{i \in I} \rho(a_i))$. If $(a_i)_{i\in I} \subseteq R(A)$ then we denote $\displaystyle \bigvee_{i \in I}^{\cdot} a_i=\rho(\bigvee_{i \in I}a_i)$. Then it is easy to prove that $(R(A), \displaystyle \bigvee^{\cdot}, \wedge, \rho(a), 1)$ is a frame \cite{Rosenthal}.

\begin{lema}
\cite{Cheptea1} If $1 \in K(A)$ then $Spec(A)=Spec(R(A))$ and $Max(A)=Max(R(A))$.
\end{lema}

\begin{lema}
\cite{Martinez} Let $A$ be a coherent quantale and $a \in A$. Then
\usecounter{nr}
\begin{list}{(\arabic{nr})}{\usecounter{nr}}
\item $\rho(a)=\bigvee \{c \in K(A)| c^k \leq a$ for some integer $k\geq 1\}$;
\item For any $c \in K(A), c \leq \rho(a)$ iff $c^k \leq a$ for some $k\geq 1$.
\end{list}
\end{lema}

\begin{lema}
\cite{Cheptea1} If $A$ is a coherent quantale then $K(R(A))=\rho(K(A))$ and $R(A)$ is a coherent frame.
\end{lema}

For any element $a$ of a coherent quantale $A$ let us consider the interval $[a)_A=\{x \in A|a \leq x\}$ and for all $x, y \in [a)_A$ denote $x \cdot_a y=xy \lor a$. Thus $[a)_A$ is closed under the multiplication $\cdot_a$ and $([a)_A, \lor, \land, \cdot_a, 0, 1)$ is a coherent quantale.

\begin{lema}
\cite{Cheptea1} The quantale $([\rho(a))_A, \lor, \land, \cdot_a, 0, 1)$ is semiprime and $Spec(A)=Spec([\rho(a))_A), Max(A)=Max([\rho(a))_A)$.
\end{lema}

Let $A, B$ be two quantales. A function $f: A \rightarrow B$ is a morphism of quantales if it preserves the arbitrary joins and the multiplication; $f$ is an integral morphism if $f(1)=1$.

\begin{lema}
\cite{Cheptea1} Let $A$ be a coherent quantale and $a \in A$.
\usecounter{nr}
\begin{list}{(\arabic{nr})}{\usecounter{nr}}
\item The function $u_a^A : A \rightarrow [a)_A$, defined by $u_a^A(x)=x \lor a$, for all $x \in A$, is an integral quantale morphism;
\item If $c \in K(A)$ then $u_a^A(c) \in K([a))$.
\end{list}
\end{lema}

Let $A$ be a quantale such that $1 \in K(A)$. For any $a \in A$, denote $D(a)=\{p \in Spec(A)|a \not\leq p\}$ and $V(a)=\{p \in Spec(A)|a \leq p\}$. Then
$Spec(A)$ is endowed with a topology whose closed sets are $(V(a))_{a \in A}$. If the quantale $A$ is algebraic then the family $(D(c))_{c \in K(A)}$ is a basis of open sets for this topology. The topology introduced here generalizes the Zariski topology (defined on the prim spectrum $Spec(R)$ of a commutation ring $R$ \cite{Atiyah}) and the Stone topology (defined on the prime spectrum $Spec_{Id}(L)$ of a bounded distributive lattice $L$ \cite{BalbesDwinger}).

Thus we denote by $Spec_Z(A)$ the prime spectrum $Spec(A)$ endowed with the topology above defined; $Max_Z(A)$ will denote the maximal spectrum $Max(A)$ considered as a subspace of $Spec_Z(A)$. According to \cite{GG}, $Spec_Z(A)$ is a {\emph{spectral space}} in the sense of \cite{Hochster}. The{\emph{flat topology}} associated to this spectral space has as basis the family of the completents of compact open subsets of  $Spec_Z(A)$(cf.\cite{Dickmann}, \cite{Johnstone}). Recall from \cite{GG} that the family $\{V(c)| c\in K(A)\}$ is a basis of open sets for the flat topology on $Spec(A)$. We shall denote by $Spec_F(A)$ this topological space.

For any $p\in Spec(A)$, let us denote $\Lambda(p) = \{q\in Spec(A)| q\leq p\}$. According to Propositions 5.6 and 5.7 of \cite{GG}, if  $p\in Spec(A)$ then the flat closure of the set $\{p\}$ is $cl_F\{p\}$ = $\Lambda(p)$ and, if $S\subseteq Spec(A)$  is compact in $Spec_Z(A)$  then its flat closure is $cl_F(S)=\displaystyle \bigcup_{p\in S} \Lambda(p)$.

An element $a\in A$ is {\emph{regular}} if it is a join of complemented elements. A maximal element in the set of proper regular elements is called {\emph{max- regular}}. The set $Sp(A)$ of max- regular elements of $A$ is called the Pierce spectrum of the quantale $A$. For any proper regular element $a$  there exists $p\in Sp(A)$ such that $a\leq p$. If $e\in B(A)$ then we denote $U(e)$ = $\{p\in Sp(A)| e\not\leq a\}$. Thus it easy to prove that the family $(U(e))_{e\in B(A)}$ is a basis of open sets for a topology on $Sp(A)$.

For any $p\in Spec(A)$ we define $s_A(p)$ = $\bigvee \{e\in B(A)| e\leq p \}$; $s_A(p)$ is regular and $s_A(p) \leq p < 1$. According to Lemma 5.8 of \cite{GG}, for each $p\in Spec(A)$, $s_A(p)$ is a max - regular element of $A$, so one obtains a function $s_A : Spec(A)\rightarrow  Sp(A)$. We know from Proposition 5.9 of \cite{GG} that $Sp(A)$ is a Boolean space and $s_A : Spec(A)\rightarrow  Sp(A)$ is surjective and continuous w.r.t. both flat and spectral topologies on $Spec(A)$. If $R$ is a commutative ring then $Sp(Id(R))$ is exactly the Pierce spectrum of $R$ (see \cite{Johnstone}, p.181).

Let $L$ be a bounded distributive lattice. For any $x \in L$, denote $D_{Id}(x)=\{P \in Spec_{Id}(L)|x \not \in P\}$ and $V_{Id}(x)=\{P \in Spec_{Id,Z}(L)|x \in P\}$. The family $(D_{Id}(x))_{x \in L}$ is a basis of open sets for the Stone topology on $Spec_{Id}(L)$; this topological space will be denoted by $Spec_{Id,Z}(L)$. Let $Max_{Id}(L)$ be the set of maximal ideals of $L$. Thus
$Max_{Id}(L) \subseteq Spec_{Id}(L)$ and $Max_{Id}(L)$ becomes a subspace of $Spec_{Id}(L)$, denoted $Max_{Id,Z}(L)$.

\section{Reticulation of a coherent quantale}

 \hspace{0.5cm}The {\emph{reticulation}} $L(A)$ of a quantale $A$ was introduced in \cite{Georgescu} as a generalization of the reticulation of a commutative ring, given in \cite{Simmons}. In \cite{Georgescu}, the reticulation $L(A)$ was characterized as a bounded distributive lattice whose prime spectrum $Spec_{Id}(L(A))$ is homeomorphic to the prime spectrum $Spec(A)$ of the quantale $A$. In this section we shall recall from \cite{Cheptea1},\cite{Georgescu} the axiomatic definition of the reticulation of the coherent quantale and some of its basic properties. Let $A$ be a coherent quantale and $K(A)$ the set of its compact elements.

\begin{definitie}
\cite{Cheptea1}
A reticulation of the quantale $A$ is a bounded distributive lattice $L$ together a surjective function $\lambda:K(A)\rightarrow L$ such that for all $a,b\in K(A)$ the following properties hold
\usecounter{nr}
\begin{list}{(\arabic{nr})}{\usecounter{nr}}
\item $\lambda(a\vee b)\leq \lambda(a)\vee \lambda(b)$;
\item $\lambda(ab)= \lambda (a)\wedge \lambda (b)$;
\item $\lambda(a)\leq\lambda(b)$ iff $a^n\leq b$ , for some integer $n\geq 1$.
\end{list}
\end{definitie}

In \cite{Cheptea1},\cite{Georgescu} there were proven the existence and the unicity of the reticulation for each coherent quantale $A$; this unique reticulation will be denoted by $(L(A),\lambda_A:K(A)\rightarrow L(A))$ or shortly $L(A)$. The reticulation $L(R)$ of a commutative ring $R$ was introduced by many authors, but the main references on this topic remain \cite{Simmons}, \cite{Johnstone}. We remark that L(R) is isomorphic to the reticulation L(Id(R)) of the quantale $Id(R)$.

\begin{lema}
\cite{Cheptea1} For all elements $a, b \in K(A)$ the following properties hold:
\usecounter{nr}
\begin{list}{(\arabic{nr})}{\usecounter{nr}}
\item $a \leq b$ implies $\lambda_A(a)\leq\lambda_A(b)$;
\item $\lambda_A(a \lor b)=\lambda_A(a) \lor \lambda_A(b)$;
\item $\lambda_A(a)=1$ iff $a=1$;
\item $\lambda_A(0)= 0$;
\item $\lambda_A(a)=0$ iff $a^n = 0$, for some integer $n\geq 1$;
\item $\lambda_A (a^n)=\lambda_A(a)$, for all integer $n\geq 1$;
\item $\rho(a)= \rho(b)$ iff $\lambda_A(a)= \lambda_A(b)$;
\item $\lambda_A(a)=0$ iff $a\leq \rho(0)$;
\item If $A$ is semiprime then $\lambda_A(a)= 0$ implies $a= 0$.
\end{list}
\end{lema}

Often the previous nine properties shall be used in the proofs without mention.

For any $a\in A$ and $I\in Id(L(A))$ let us denote $a^{\ast}= \{\lambda_A(c)|c\in K(A), c\leq a\}$ and $ I_{\ast} = \bigvee\{c\in K(A)|\lambda_A(c)\in I\}$.

\begin{lema}
\cite{Cheptea1} The following assertions hold
\usecounter{nr}
\begin{list}{(\arabic{nr})}{\usecounter{nr}}
\item If $a\in A $ then $a^{\ast}$ is an ideal of $L(A)$ and $a\leq (a^{\ast})_{\ast}$;
\item If $I\in Id(L(A))$ then $(I_{\ast})^{\ast}=I$;
\item If $p\in Spec(A)$ then  $(p^{\ast})_{\ast}= p$ and $p^{\ast}\in Spec_{Id}(L(A))$;
\item If $P\in Spec_{Id}((L(A))$ then $P_{\ast}\in Spec(A)$;
\item If $p\in K(A)$ then $c^{\ast}= (\lambda_A(c)]$;
\item If $c\in K(A)$ and $I\in Id(L(A))$ then $c\leq I_{\ast}$ iff $\lambda_A(c)\in I$;
\item If $a\in A$ and $I\in Id(L(A))$ then $\rho(a)=(a^{\ast})_{\ast}$, $a^{\ast}= (\rho(a))^{\ast}$ and $\rho(I_{\ast})= I_{\ast}$;
\item If $c\in K(A)$ and $p\in Spec(A)$ then $c\leq p$ iff $\lambda_A(c) \in p^{\ast}$.
\end{list}
\end{lema}

\begin{lema}
Let $A$ be a coherent quantale. The following assertions hold
\usecounter{nr}
\begin{list}{(\arabic{nr})}{\usecounter{nr}}
\item If $a,b\in A $ then $(ab)^{\ast}$ = $(a\land b)^{\ast}$ = $a^{\ast}\bigcap b^{\ast}$;
\item If $(a_i)_{i\in I}$ is a family of elements of $A$ then    $(\displaystyle \bigvee_{i\in I} a_i)^{\ast}$ =  $\displaystyle \bigvee_{i\in I} a_i^{\ast}$.
\end{list}
\end{lema}

\begin{proof} First we remark that $(ab)^{\ast}\subseteq  (a\land b)^{\ast}\subseteq  a^{\ast}\bigcap b^{\ast}$. In order to prove that $a^{\ast}\bigcap b^{\ast}\subseteq (ab)^{\ast}$, let us assume that $x\in  a^{\ast}\bigcap b^{\ast}$. Then $x = \lambda_A(c) = \lambda_A(d)$, for some compact elements $c,d$ that verify the properties $c\leq a$ and $d\leq b$. Thus one gets $\lambda_A(cd)\leq \lambda_A(ab)$, so there exists a positive integer $n$ such that $c^nd^n\leq ab$. Therefore $x = \lambda_A(c^nd^n) \leq \lambda_A(ab)$, hence it follows that $x\in (ab)^{\ast}$. The property (2) follows similarly.

\end{proof}

According to Lemma 3.3, one can consider the following order- preserving functions: $ u:Spec(A)\rightarrow Spec_{Id}(L(A))$ and $ v:Spec_{Id}(L(A))\rightarrow Spec(A)$, defined by $u(p) = p^{\ast}$ and $v(P) = P_{\ast}$, for all $p\in Spec(A)$ and $P\in Spec_{Id}(L(A))$.

\begin{lema}
\cite{Cheptea1}
The functions $u$ and $v$ are homeomorphisms, inverse to one another.
\end{lema}
\begin{corolar}
$Max_Z(A)$ and $Max_{Id,Z}(L(A))$ are homeomorphic.
\end{corolar}
\begin{propozitie}
\cite{Cheptea1}
The functions $\Phi:R(A) \rightarrow Id(L(A))$ and $\Psi:Id(L(A))\rightarrow R(A)$ defined by $\Phi(a)= a^{\ast}$ and $\Psi(I)= I_{\ast}$, for all $a\in R(A)$ and $I \in Id(L(A))$, are frame isomorphisms, inverse to one another.
\end{propozitie}

 \hspace{0.5cm}The {\emph{Boolean center}} of an arbitrary quantale $A$ is the Boolean algebra $B(A)$ of complemented elements of $A$ (cf. \cite{Birkhoff},\cite{Jipsen}). The following lemma collects some elementary properties of the elements of $B(A)$.
\begin{lema}
\cite{Birkhoff},\cite{Jipsen} Let $A$ be a quantale and $a,b \in A$, $e\in B(A)$. Then the following properties hold:
\usecounter{nr}
\begin{list}{(\arabic{nr})}{\usecounter{nr}}
\item $ a\in B(A)$ iff $a \lor a^{\bot} = 1$;
\item $ a\land b$ = $ae$;
\item $e\rightarrow a$ = $e^{\bot}\lor a$;
\item If $a\lor b = 1$ and  $ab = 0$, then $a,b \in B(A)$;
\item $(a \land b) \lor e = (a\lor e)\land (b\land e)$;
\item  For any integer $n\geq1$, $a\lor b = 1$ and $a^n b^n = 0$ implies $a^n,b^n \in B(A)$.
\end{list}
\end{lema}

\begin{lema}
\cite{Cheptea1} If $ 1\in K(A)$ then $ B(A)\subseteq K(A)$.
\end{lema}
For a bounded distributive lattice $L$ we shall denote by $B(L)$ the Boolean algebra of the complemented elements of $L$. It is well-known that $B(L)$ is isomorphic  to the Boolean center $B(Id(L))$ of the frame $Id(L)$ (see \cite{Birkhoff}, \cite{Johnstone}, \cite{GCM}).

Let us fix a coherent quantale $A$.

\begin{lema}
\cite{GG}
Assume $c\in K(A)$. Then $\lambda_A(c) \in B(L(A))$ if and only if $c^n \in B(A)$, for some integer $n\geq 1$.
\end{lema}

\begin{corolar}
\cite{Cheptea1} The function $\lambda_A|_{B(A)} :B(A)\rightarrow B(L(A))$ is a Boolean isomorphism.
\end{corolar}

If $L$ is bounded distributive lattice and $I\in Id(L)$ then the {\emph{annihilator}} of $I$ is the ideal $Ann(I) = \{ x \in L |x \land y =0$, for all $y\in L \}$.

\begin{lema}
If $c\in K(A)$ and $p\in Spec(A)$ then $Ann(\lambda_A(c))\subseteq p^{\ast}$ if and only if $c\rightarrow \rho(0)\leq p$.
\end{lema}

The next two propositions concern the behaviour of reticulation w.r.t. the annihilators.
\begin{propozitie}
If $a$ is an element of a coherent quantale then $Ann(a^\ast)=(a \rightarrow \rho(0))^\ast$; if $A$ is semiprime then $Ann(a^\ast)=(a^\perp)^\ast$.
\end{propozitie}

\begin{propozitie}
Assume that $A$ is a coherent quantale. If $I$ is an ideal of $L(A)$ then $(Ann(I))_\ast=I_\ast \rightarrow \rho(0)$; if $A$ is semiprime then
$(Ann(I))_\ast=(I_\ast)^\perp$.
\end{propozitie}

If $A$ is a quantale then we denote by $Min(A)$ the set of {\emph{minimal $m$ - prime}} elements of $A$; $Min(A)$ is called the minimal prime spectrum of $A$. If $1\in K(A)$ then for any $p\in Spec(A)$ there exists $q\in Min(A)$ such that $q\leq p$.

\begin{propozitie}
Let $A$ be a coherent quantale. If $p\in Spec(A)$ then the following are equivalent:
\usecounter{nr}
\begin{list}{(\arabic{nr})}{\usecounter{nr}}

\item $p\in Min(A)$;

\item For all $c\in K(A))$, $c\leq p$ if and only if $c\rightarrow \rho(0)\not\leq p$.
\end{list}
\end{propozitie}

\begin{corolar}

If $A$ is semiprime coherent quantale and $p\in Spec(A)$ then $p\in Min(A)$ if and only if for all $c\in K(A)$, $c\leq p$ implies $c^{\perp}\not\leq p$.

\end{corolar}

\section{Pure and w - pure elements in a quantale}

 The pure elements in a quantale extend the pure ideals of a ring \cite{Lam},\cite{SimmonsC} and the $\sigma$ -ideals of a bounded distributive lattice \cite{Cornish1}, \cite{GeorgescuVoiculescu}. More precisely, an ideal $I$ of bounded distributive lattice $L$ is a {\emph{$\sigma$ - ideal}} if for all $x\in I$, we have $I\lor Ann(x) = L$. An ideal $I$ of a commutative ring $R$ is {\emph{pure}} (or {\emph{virginal}}, in terminology of \cite{Borceux}) if for all  $x\in I$, we have $I\lor Ann(x) = L$. The notions of pure ideals and $\sigma$ -ideals have been generalized to various abstract structures : frames \cite{Johnstone}, quantales \cite{PasekaRN}, multiplicative - ideal structures \cite{GeorgescuVoiculescu}, two - side carriers \cite{Simmons}, etc. In this way appeared the abstract notion of pure element.

An element $a$ of an arbitrary algebraic quantale $A$ is said to be {\emph{pure}} (or {\emph{virginal}}, in the terminology of  \cite{GeorgescuVoiculescu2}) if for all $c\in K(A)$, $c\leq a$ implies $a \lor c^{\perp} = 1$. Then an ideal $I$ of a commutative ring $R$ is pure if and only if $I$ is a pure element of the quantale $Id(R)$. An ideal $I$ of a bounded distributive lattice $L$ is a $\sigma$ - ideal if and only if $I$ is a pure element of the frame $Id(L)$. The set of pure elements of the quantale $A$ will be denoted by $Vir(A)$.

The quantales are {\emph{multiplicative - ideals structures}} in sense of \cite{GeorgescuVoiculescu2}, so all the results of this paper remain true for the pure elements in a quantale.

\begin{lema}
\cite{GeorgescuVoiculescu2}
If $A$ is an algebraic quantale then the following hold:
\usecounter{nr}
\begin{list}{(\arabic{nr})}{\usecounter{nr}}
\item If $a\in A$ is pure then $a$ = $\bigvee \{c\in K(A)| a\lor c^{\perp} = 1\}$;
\item If $a,b\in A$ are pure the $ab = a\land b$;
\item If $(a_i)_{i\in I}$ is a family of pure elements then $\displaystyle \bigvee_{i\in I} a_i$ is pure;
\item The structure $(Vir(A),\bigvee, \land, 0,1)$ is a frame.
\end{list}
\end{lema}

Keeping the notations from \cite{Borceux}, \cite{GeorgescuVoiculescu2}, for any  $a\in A$ we define the following elements of $A$:

$O(a)$ = $\bigvee \{u\in A | uv = 0, for  some   v\in A, v\not\leq a\}$,

$Ker(a)$ = $\bigvee \{c\in K(A) | c\leq a, a\lor c^{\perp} = 1\}$;

 $Vir(a)$ = $\bigvee \{b\in Vir(A) | b\leq a \}$.

 It is easy to show that $Ker(a)$ = $\bigvee \{c\in K(A) |  a\lor c^{\perp} = 1\}$, because $a\lor c^{\perp} = 1$ implies that $c\leq a$. In general we have $Vir(a)\leq Ker(a)\leq a$ and $a$ is pure if and only if $a = Vir(a)$. If $A$ is an algebraic quantale then

 $O(a)$ = $\bigvee \{c\in K(A) | cd = 0, for  some   d\in K(A), d\not\leq a\}$.

 For any $p\in Spec(A)$ we have $O(p)$ = $\bigvee \{c\in K(A) | c^{\perp}\not\leq p\}$, hence $O(p)\leq p$.

\begin{lema}
\cite{GeorgescuVoiculescu2}
Let $A$ be is algebraic quantale such that $1\in K(A)$. Then the following hold:
\usecounter{nr}
\begin{list}{(\arabic{nr})}{\usecounter{nr}}
\item If $m\in Max(A)$ then $O(m) = Ker(m)$;
\item If $a\in Vir(A)$ then

$a = Vir(\bigwedge \{m\in Max(A)|a\leq m \})$ = $\bigwedge \{Vir(m)|m\in Max(A),a\leq m \}$;

\item The map $\rho:Vir(A) \rightarrow R(A)$ is an injective frame morphism, left adjoint to $Vir:R(A) \rightarrow Vir(A)$;
\item $Vir(A)$ is a spatial frame and $Vir: Spec(A) \rightarrow Spec(Vir(A))$ is a continuous map.
\end{list}
\end{lema}

The following lemma improves the property $(2)$ from Lemma 4.1.

\begin{lema}
Let $A$ be is an algebraic quantale such that $1\in K(A)$. If $a\in A$ and $b\in VA$ then $a\land b$ = $ab$.

\end{lema}

\begin{proof} Assume that $c$ is a compact element of $A$ such that $c\leq a\land b$. From $c\leq b$ one gets $b\lor c^{\perp} = 1$, hence $c = c(b\lor c^{\perp}) = cb$. Thus $c\leq ab$, hence $a\land b \leq ab$. The converse inequality is obvious.

\end{proof}

\begin{lema}
Let $A$ be is an algebraic quantale such that $1\in K(A)$. If $p$ is an $m$ - of prime element $A$ then $Vir(p) = Vir(O(p))$.

\end{lema}

\begin{proof} Let $p$ be an element of $Spec(A)$. Since $O(p)\leq p$ then the inequality $Vir(O(p))\leq Vir(p)$ holds. In order to obtain the inverse inequality $Vir(p)\leq Vir(O(p))$, we must prove that for any compact element $c$ of $A$, $c\leq Vir(p)$ implies $c\leq Vir(O(p))$. If $c\leq Vir(p)$ then $Vir(p)\lor c^{\perp} = 1$, so $p\lor c^{\perp} = 1$. Thus $p\lor d = 1$ for some $d\in K(A)$ such that $d\leq c^{\perp}$, hence $d\not\leq p$  and $cd = 0$. It follows that $c\leq O(p)$, therefore  $Vir(p)\leq O(p)$. The last inequality implies  $Vir(p)\leq Vir(O(p))$.

\end{proof}

\begin{propozitie} Let $A$  a semiprime algebraic quantale such that $1\in K(A)$. If $a\in Vir(A)$ then $\rho(a) = a$.

\end{propozitie}

\begin{proof} Assume that $c\in K(A)$ and $n$ is a positive integer such that $c^n\leq a$. Then $a\lor (c^n)^{\perp} = 1$, so there exists $d\in K(A)$ such that $d\leq (c^n)^{\perp}$ and $a\lor d = 1$. It follows that $c^nd = 0$, hence $\lambda_A(cd)$ = $\lambda_A(c)\land \lambda_A(d)$ = $\lambda_A(c^nd) = 0$. Since $A$ is semiprime one gets $cd = 0$, therefore $c = c(a\lor d) = ca$. By using Lemma 4.3 one obtains $c = c\land a$, hence $c\leq a$. According to Lemma 2.4 it results that $\rho(a) = a$.

\end{proof}

An element $a$ of a quantale $A$ is said to be {\emph{weakly - pure}} ( = {\emph{w - pure}}) if for all $c\in K(A)$, $c\leq a$ implies $a\lor (c\rightarrow \rho(0)) = 1$.

\begin{lema} Any pure element $a$ of $A$ is w - pure.

\end{lema}

\begin{proof} For any compact element $c\leq a$ we have $c^{\perp} = c\rightarrow 0 \leq c\rightarrow \rho(0)$, so $1 = a\lor c^{\perp}\leq a\lor (c\rightarrow \rho(0))$, hence $a$ is w - pure.

\end{proof}

If $A$ is semiprime then $a\in A$ is pure if and only if it is w - pure. We denote by $Vir_w(A)$ the set of w - pure elements of $A$. By the previous lemma we have $Vir(A)\subseteq Vir_w(A)$.

\begin{lema}
Let $A$ be an algebraic quantale such that $1\in K(A)$.
\usecounter{nr}
\begin{list}{(\arabic{nr})}{\usecounter{nr}}
\item $Vir_w(A)$ is closed under $\cdot$ and $\land$;
\item For any family $(a_i)_{i\in I}\subseteq Vir_w(A)$ we have $\displaystyle \bigvee_{i \in I}a_i \in Vir_w(A)$.

\end{list}
\end{lema}

\begin{proof}
(1) Assume that $a,b$ are two w - pure elements of $A$. If $c$ is a compact element of $A$ such that $c\leq ab$ then $c\leq a$ and $c\leq b$, so we have $a\lor (c\rightarrow \rho(0) =  b\lor (c\rightarrow \rho(0) = 1$. By Lemma 2.1,(3) we get $(ab)\lor (c\rightarrow \rho(0)) = 1$, so $ab$ is w - pure. Similarly, $c,d\in K(A)$ and $c\leq a\land b$ implies  $(a\land b)\lor (c\rightarrow \rho(0)) = 1$, hence $a\land b$ is w - pure.

(2) Let us denote $a = \displaystyle \bigvee_{i \in I}a_i $. Assume that $c$ is a compact element of $A$ such that $c\leq a$, hence $c\leq \displaystyle \bigvee_{i \in J}a_i $, for some finite subset $J$ of $I$. Thus there exist the compact elements $d_i, i\in J$ such that $c\leq \displaystyle \bigvee_{i \in J}d_i $ and $d_i\leq a_i$, for all $i\in J$. For any $i\in J$ we have $a\lor (d_i\rightarrow \rho(0))\geq a_i\lor (d_i\rightarrow \rho(0)) = 1$, so, by using Lemma 2.1,(3) it follows that $a\lor \displaystyle \bigwedge_{i \in J}(d_i\rightarrow \rho(0))$ = $1$. Observing that $c\rightarrow \rho(0)\geq  (\displaystyle \bigvee_{i \in J}d_i)\rightarrow \rho(0) = \displaystyle \bigwedge_{i \in J}(d_i\rightarrow \rho(0)) $, one obtains the inequality $a\lor (\rightarrow \rho(0))\geq a\lor \displaystyle \bigwedge_{i \in J}(d_i\rightarrow \rho(0)) $ = $1$. Thus $a\lor (c\rightarrow \rho(0)) = 1$, hence $a$ is w- pure.

\end{proof}

\begin{lema}
If an element $a$ of coherent quantale $A$ is w - pure then $a^{\ast}$ is a $\sigma$- ideal of the reticulation $L(A)$. Particularly, if $a$ is pure then $a^{\ast}$ is a $\sigma$- ideal.

\end{lema}

\begin{proof}
Assume that $x\in a^{\ast}$, hence $x = \lambda_A(c)$ for some compact element $c$ of $A$, such that $c\leq a$. Then $a\lor (c\rightarrow \rho(0)) = 1$, so there exist $c,d\in K(A)$ such that $d\leq a$, $e\leq c\rightarrow \rho(0)$ and $d\lor e = 1$. It follows that $\lambda_A(d)\in a^{\ast}$, $ec\leq \rho(0)$ and $\lambda_A(d)\lor \lambda_A(e)$ = $\lambda_A(d\lor e) = 1$. On the other hand we have  $\lambda_A(d)\land \lambda_A(e)$ = $\lambda_A(de)\leq \lambda_A(\rho(0)) = 0$, so $\lambda_A(e)\in Ann(\lambda_A(c))$. Thus $a^{\ast}\lor Ann(x)$ = $a^{\ast}\lor Ann(\lambda_A(c))$ = $L(A)$, so $a^{\ast}$ is a $\sigma$- ideal of $L(A)$. The second part of proposition follows by Lemma 4.5.

\end{proof}

\begin{lema}
Let $A$ be a coherent quantale and $J$ a $\sigma$ - ideal of $L(A)$. Then $J_{\ast}$ is a w - pure element of $A$.
\end{lema}

\begin {proof} Let $c$ be a compact element of $A$ such that $c\leq J_{\ast}$. By Lemma 3.3,(6) we have $\lambda_A(c)\in J$, hence  $J\lor Ann(\lambda_A(c)) = L(A)$. Then there exist $d,e\in K(A)$ such that $\lambda_A(c)\in J$, $\lambda_A(e)\in Ann(\lambda_A(c)$ and $\lambda_A(d\lor e)$ = $\lambda_A(d)\lor \lambda_A(e) = 1$. By Lemmas 3.3,(6) and 3.2,(3) we obtain $d\leq J_{\ast}$ and $d\lor e = 1$. From $\lambda_A(e)\in Ann(\lambda_A(c)$ we get $\lambda_A(ce)$ = $\lambda_A(c)\land \lambda_A(e) = 0$, hence $ce\leq \rho(0)$ (cf. Lemma 3.2,(8)). It follows that $e\leq c\rightarrow \rho(0)$, so $1 = c\lor d\leq J_{\ast}\lor (c\rightarrow \rho(0))$. Thus $J_{\ast}\lor (c\rightarrow \rho(0)) = 1$, hence $J_{\ast}$ is w - pure.

\end{proof}

\begin{corolar} Let $A$ be a semiprime quantale. If $J$ is a $\sigma$ - ideal of $L(A)$ then $J_{\ast}$ is a pure element of $A$.

\end{corolar}

\begin{corolar}
Let $A$ be a coherent quantale.
\usecounter{nr}
\begin{list}{(\arabic{nr})}{\usecounter{nr}}
\item If $a\in A$ is w - pure then $\rho(a)$ is w - pure;
\item An ideal $J$ of $L(A)$ is a $\sigma$ - ideal if and only if $J_{\ast}$ is a w - pure element of $A$.
\end{list}
\end{corolar}

\begin{proof}
(1) If $a\in A$ is w - pure then $a^{\ast}$ is a $\sigma$ - ideal of $L(A)$ (cf. Lemma 4.8). By applying Lemma 4.9 it follows that $\rho(a) = (a^{\ast})_{\ast}$ is $w$ - pure.

(2) We apply Lemmas 4.8, 4.9 and 3.3,(2).

\end{proof}

Following \cite{Cornish1}, \cite{GeorgescuVoiculescu}, if $I$ is an ideal of a bounded distributive lattice then we denote $\sigma(I)$ = $\{x\in L| I\lor Ann(x) = L \}$. We remark that in the frame $Id(L)$ we have $\sigma(I)$ = $Ker(I)$.

\begin{lema}
Let $A$ be a coherent quantale. Then for all $a\in A$ and $c\in K(A)$, the following equivalence holds: $c\leq Ker(a)$ if and only if $a\lor c^{\perp} = 1$.
\end{lema}

\begin{proof} If $c\leq Ker(a)$ then there exist $d_1,...,d_n\in K(A)$ such that $c\leq d_1\lor...\lor d_n$ and $d_i\leq a$, $a\lor d^{\perp}_i = 1$, for all $i = 1,...,n$. Denoting $d = \displaystyle \bigvee_{i=1}^n d_i$ we have $d\in K(A)$ and, by using Lemma 2.1,(3), the following equalities hold:

$a\lor d^{\perp}$ = $a\lor (\displaystyle \bigvee_{i=1}^n d_i )^{\perp}$ = $a\lor \displaystyle \bigwedge_{i=1}^n  d_i^{\perp} = 1$.

Since $c\leq d$ it follows that $d^{\perp}\leq c^{\perp}$, so $a\lor c^{\perp} = 1$. The proof of converse implication is obvious.

\end{proof}

The following two theorems emphasize the way in wich the reticulation preserves the operator $Ker(\cdot)$.

\begin{teorema}
Let $A$ be a coherent quantale.
\usecounter{nr}
\begin{list}{(\arabic{nr})}{\usecounter{nr}}
\item For any $a\in A$ we have $(Ker(a))^{\ast} = \sigma(a^{\ast})$;

\item For any ideal $I$ of $L(A)$ we have $(\sigma(I))_{\ast} = Ker(a_{\ast})$.

\end{list}
\end{teorema}

\begin{proof}
(1) Assume that $x\in \sigma(a^{\ast})$, so $a^{\ast}\lor Ann(x) = L(A)$. Then there exists $c\in K(A)$ such that $x = \lambda_A(c)$, so $a^{\ast}\lor Ann(\lambda_A(c)) = L(A)$. Then one can find $d,e\in K(A)$ such that $\lambda_A(d)\in a^{\ast}$, $\lambda_A(e)\in Ann(\lambda_A(c))$ and $\lambda_A(d\lor e)$ = $\lambda_A(d)\lor \lambda_A(e) = 1$. Thus $\lambda_A(ec)$ = $\lambda_A(e)\land \lambda_A(c) = 0$, so there exists a positive integer n such that $e^nc^n = 0$, hence $e^n\leq (c^n)^{\perp}$. According to Lemma 3.2,(3) have $d\lor e = 1$, so $d^n\lor e^n = 1$ (by Lemma 2.2,(2)). On can take $d\leq a$, hence $a\lor (c^n)^{\perp} = 1$. By Lemma 4.12 we get $c^n\leq Ker(a)$, so $x = \lambda_A(c) = \lambda_A(c^n)\in (Ker(a))^{\ast}$. We have proven that  $\sigma(a^{\ast})\subseteq (Ker(a))^{\ast}$.

In order to prove that $(Ker(a))^{\ast}\subseteq \sigma(a^{\ast})$, let us assume that $x\in (Ker(a))^{\ast}$, so $x = \lambda_A(c)$, for some compact element $c$ having the property $c\leq Ker(a)$. By Lemma 4.12 we have $a\lor c^{\perp} = 1$, hence, by using Lemma 3.4,(2),  one obtains $a^{\ast}\lor (c^{\perp})^{\ast} = L(A)$. According to Proposition 3.13, $Ann(x)$ = $Ann(\lambda_A(c))$ = $(c\rightarrow \rho(0))^{\ast}$, therefore the inequality $c^{\perp} = c\rightarrow 0\leq c\rightarrow \rho(0)$ implies $(c^{\perp})^{\ast}\subseteq (c\rightarrow \rho(0))^{\ast} = Ann(x)$. It follows that $a^{\ast}\lor Ann(x) = L(A)$, i.e $x\in \sigma(a^{\ast})$

(2) Assume that $I$ is an ideal of $L(I)$ and $c$ is a compact element of $A$. If $I_{\ast}\lor c^{\perp} = 1$ then, by using  Lemmas 3.3,(2) and 3.4,(2),  one obtains $I\lor (c^{\perp })^{\ast}$= $(I_{\ast})^{\ast}\lor (c^{\perp})^{\ast}$ = $(I_{\ast}\lor c^{\perp})^{\ast} = L(A)$. Conversely, assuming that $I\lor (c^{\perp })^{\ast} = L(A)$, one infers that there exist $d,e\in K(A)$ such that $\lambda_A(d)\in I$, $\lambda_A(e)\in (c^{\perp })^{\ast}$ and $\lambda_A(d\lor e)$ = $\lambda_A(d)\lor \lambda_A(e) = 1$. We remark that $\lambda_A(d)\in I$ implies $d\leq I_{\ast}$ and $\lambda_A(d\lor e) = 1$ implies $d\lor e = 1$. It is obvious that one can assume that $e\leq c^{\perp}$, so $I_{\ast}\lor c^{\perp} = 1$. Thus we get the following equivalence: $I_{\ast}\lor c^{\perp} = 1$ if and only if  $I\lor (c^{\perp })^{\ast} = L(A)$. By using this equivalence and  Lemma 4.12, the following  hold: $c\leq Ker(I_{\ast})$ iff $I\lor c^{\perp} = 1$ iff $I\lor (c^{\perp })^{\ast} = L(A)$  iff $I\lor Ann(\lambda_A(c)) = L(A)$  iff $\lambda_A(c)\in \sigma(I)$ iff $c\leq (\sigma(I))_{\ast}$. The previous equivalences hold for all $c\in K(A)$, therefore one gets the equality $(\sigma(I))_{\ast} = Ker(a_{\ast})$.

\end{proof}

\begin{teorema}
If $a$ is a w - pure radical element of a coherent quantale $A$ then $a = \rho(Ker(a))$ and $Ker(a) = Vir(a)$.
\end{teorema}

\begin{proof}
First we shall show that $a = \rho(Ker(a))$. From $Ker(a)\leq a$ we get $\rho(Ker(a))\leq \rho(a)\leq a$. In order to prove that $a\leq \rho(Ker(a))$, assume that $c$ is a compact element of $A$ such that $c\leq a$. Since $a$ is w - pure, we get $a\lor (c\rightarrow \rho(0)) = 1$. By the compactness of $1$, there exists $d\in K(A)$ such that $d\leq c\rightarrow \rho(0)$ and $a\lor d = 1$. According to Lemma 2.4,(2), from $dc\leq \rho(0)$ one obtains $d^nc^n = 0$ for some positive integer $n$, hence $d^n\leq (c^n)^{\perp}$. Applying Lemma 2.1,(3), from $a\lor d = 1$ we obtain $a\lor d^n = 1$,
hence $a\lor (c^n)^{\perp} = 1$. Thus the compact element $c^n$ verifies the inequality $c^n\leq Ker(a)$ (cf. Lemma 4.1), hence, by using Lemma 2.4,(2) one obtains $c\leq \rho(Ker(a))$. It follows that $a\leq \rho(Ker(a))$, hence $a = \rho(Ker(a))$.

In order to show that $Ker(a) = Vir(a)$, we recall that $Ker(a)\leq Vir(a)$, hence it suffices to check that $Ker(a)$ is pure. Assume that $c$ is a compact element such that $c\leq Ker(a)$, so $a\lor c^{\perp} = 1$ (cf. Lemma 4.1). We have proven that $a = \rho(Ker(a))$, so $\rho(Ker(a))\lor c^{\perp} = 1$. By the compactness of $1$ there exists $d\in K(A)$ such that $d\leq \rho(Ker(a))$ and $d\lor  c^{\perp} = 1$. According to Lemma 2.4,(1), there exists a positive integer n such that $d^n\leq Ker(a)$ and, by Lemma 4.1, on gets $d^n\lor c^{\perp} = 1$. It follows that $Ker(a)\lor c^{\perp} = 1$, so $Ker(a)$ is pure.
\end{proof}

\begin{lema}
If $a$ is a pure element of a coherent quantale $A$ then $a = Vir(\rho(a))$.
\end{lema}

\begin{proof}
We remark that $a\leq \rho(a)\leq \bigwedge \{m\in Max(A)| a\leq m \}$. In accordance with Lemma 4.2,(3) it follows that the following hold: $a = Vir(a)\leq Vir(\rho(a))\leq Vir(\bigwedge \{m\in Max(A)| a\leq m \}) = a$, so $a = Vir(\rho(a))$.

\end{proof}

According to Lemma 4.7, one can consider the function $w: Vir(A)\rightarrow Vir(Id(L(A)))$, defined by $w(a) = a^{\ast}$, for all $a\in Vir(A)$. By using Lemma 3.4,(2) it follows that $w$ is a frame morphism.

By using the previous results one can obtain a new proof of a theorem given in \cite{Georgescu}.

\begin{teorema}
The map $w$ is a frame isomorphism.

\end{teorema}

\begin{proof}
According to Lemma 4.8 one can consider the composition $Vir\circ (\cdot )_{\ast}$ of the following functions:

$Vir(Id(L(A))) \xrightarrow[]{(\cdot)_\ast} Vir_w(A) \xrightarrow[]{Vir} Vir(A)$.

We shall prove that $Vir\circ (\cdot )_{\ast}$ is the inverse function of $w$. Let $a$ be an arbitrary pure element of $A$. By using Lemmas 4.11 and 3.3(7), the following equalities hold: $a = Vir(\rho(a)) = Vir((a^{\ast})_{\ast}) = (Vir\circ (\cdot )_{\ast})(w(a))$.

If $J\in V(Id(L(A)))$ then $a = J_{\ast}$ is a w - pure radical element of $A$ (cf. Lemma 4.8). Then by using Theorem 4.3 and Lemma 3.3, the following equalities hold:

$w(Vir(J_{\ast})) = (Vir(a))^{\ast} = (Ker(a))^{\ast} = (\rho(Ker(a)))^{\ast} = a^{\ast} = (J_{\ast})^{\ast} = J$.

\end{proof}

Let $A$ be a semiprime coherent quantale. By Proposition 4.5 we have $Vir(A)\subseteq R(A)$. Therefore, by using the proof of the previous theorem it follows that the functions

$(\cdot )^{\ast}: Vir(A)\rightarrow Vir(Id(L(A))))$, $(\cdot)_{\ast}: Vir(Id(L(A)))\rightarrow Vir(A)$

 are the inverse frame morphisms that give that the frames $Vir(A)$ and $Vir(Id(L(A))))$ are isomorphic.

\begin{lema}
Let $A$ be a coherent quantale.
\usecounter{nr}
\begin{list}{(\arabic{nr})}{\usecounter{nr}}
\item For any $p\in Spec(A)$ we have $O(p) = \bigvee \{c\in K(A)|c\leq p, c^{\perp}\not\leq p\}$;

\item For any $p\in Spec(A)$, $c\leq O(p)$ if and only if $c^{\perp}\not\leq p$.

\end{list}
\end{lema}

For any $p\in Spec(A)$, define $\tilde O(p)$ = $\bigvee \{c\in K(A)|c\leq p, c\rightarrow \rho(0)\not\leq p\}$. It is easy to prove that $\tilde O(p)\leq O(p)$.

\begin{lema}
For all $p\in Spec(A)$ and $c\in K(A)$, $c\leq \tilde O(p)$ if and only if $c\leq p$ and $c\rightarrow \rho(0)\not\leq p$.
\end{lema}

\begin{proof}

Assume that $c\leq \tilde O(p)$, so there exist $d_1,...,d_n\in K(A)$ such that $c\leq \displaystyle \bigvee_{i=1}^n d_i$,  $d_i\leq p$ and $d_i\rightarrow \rho(0)\not\leq p$, for each $i = 1,...,n$. Then $c\leq p$ and $( \displaystyle \bigvee_{i=1}^n d_i)\rightarrow \rho(0)$ = $\displaystyle \bigwedge_{i=1}^n (d_i\rightarrow \rho(0))\not\leq p$ (because $p$ is $m$ - prime). The converse implication is obvious.

\end{proof}

\begin{teorema}
Let $A$ be a coherent quantale.
\usecounter{nr}
\begin{list}{(\arabic{nr})}{\usecounter{nr}}
\item For any $p\in Spec(A)$ we have $O(p^{\ast}) = (\tilde O(p))^{\ast}$;

\item For any $P\in Spec_{Id}(L(A))$ we have $(O(P))_{\ast}$ = $\tilde O(P_{\ast})$.

\end{list}
\end{teorema}

\begin{proof}
(1) Assume that $x\in O(p^{\ast})$, hence $Ann(x)\not\subseteq p^{\ast}$ and $x\in p^{\ast}$. Let us take a compact elemenent $c$ such that $c\leq p$ and $x = \lambda_A(c)$, so $Ann(\lambda_A(c))\not\subseteq p^{\ast}$. By Lemma 3.12 we have $c\rightarrow \rho(0)\not\leq p$, so $c\leq \tilde O(p)$. It follows that $x = \lambda_A(c)\in (\tilde O(p))^{\ast}$, so one obtains the inclusion  $O(p^{\ast})\subseteq (\tilde O(p))^{\ast}$.

In order to prove that $(\tilde O(p))^{\ast}\subseteq p^{\ast}$, let us assume that $x\in (\tilde O(p))^{\ast}$, hence there exists $c\in K(A)$ such that $c\leq \tilde Q(p)$ and $x = \lambda_A(c)$. By Lemma 4.18 we have $c\leq p$ and $c\rightarrow \rho(0)\not\leq p$. By using Lemma 3.12 one gets $Ann(\lambda_A(c))\not\subseteq p^{\ast}$, hence $x = \lambda_A(c)\in O(p^{\ast})$.

(2) Let $c$ be a compact element of $A$ such that $c\leq (O(P))_{\ast}$, hence $\lambda_A(c)\in O(P)$ (cf. Lemma 3.3,(6)). It follows that $\lambda_A(c)\in P$ and $Ann(\lambda_A(c))\not\subseteq P = (P_{\ast})^{\ast}$. By using Lemma 3.12  we get $c\rightarrow \rho(0)\not\leq P_{\ast}$, therefore by applying Lemma 4.18 we have $c\leq \tilde O(P_{\ast})$. We conclude that $(O(P))_{\ast}\leq \tilde O(P_{\ast})$.

In order to prove the converse inequality $\tilde O(P_{\ast})\leq (O(P))_{\ast}$,  let us assume that $c\in K(A)$ and $c\leq  \tilde O(P_{\ast})$. By Lemma 4.18 we have $c\leq P_{\ast}$ and $c\rightarrow \rho(0)\not\leq P_{\ast}$. Applying Lemma  3.12 we obtain $Ann(\lambda_A(c))\not\leq (P_{\ast})^{\ast} = P$. From $c\leq P_{\ast}$ we get $\lambda_A(c)\in P$ (according to Lemma 3.3,(6)). Therefore $\lambda_A(c)\in O(P)$, so we conclude that $c\leq (O(P))_{\ast}$.

\end{proof}

\begin{corolar}
Let $A$ be a semiprime coherent quantale.
\usecounter{nr}
\begin{list}{(\arabic{nr})}{\usecounter{nr}}
\item For any $p\in Spec(A)$ we have $(O(p^{\ast}) = (O(p))^{\ast}$;
\item For any $P\in Spec_{Id}(L(A))$ we have $(O(P))_{\ast}$ = $O(P_{\ast})$.

\end{list}

\end{corolar}

Theorem 4.19 and Corollary 4.20 show us how the reticulation preserves the operator $O(\cdot)$. They will be used many-times in the proofs of the theorems in the next sections (see e.g. the proof of Theorem 5.10).

\section{Further properties of pure elements}

\begin{lema} Let $A$ be an algebraic quantale. Any regular element of $A$ is pure.

\end{lema}

\begin{proof} Let $e$ be a complemented element of $A$. If $c$ is a compact element of $A$ such that $c\leq e$ then $e^{\perp} \leq d^{\perp}$, therefore $1 = e\lor e^{\perp}\leq e\lor c^{\perp}$. Then $e$ is pure.

Let now $q$ be a regular element of $A$ and $e\in B(A)$ such that $e\leq q$. We have proven that any complemented element is pure, so $e = Vir(e)\leq Vir(p)$. It follows that $q = \bigvee \{e\in B(A)| e\leq q\} \leq Vir(q)$, so $q = Vir(q)$. Therefore $q$ is pure.

\end{proof}

For any $p\in Spec(Vir(A))$ let us define $t_A(p) = \bigvee \{e\in B(A)| e\leq q\}$. Thus $t_A(p)$ is regular and $t_A(p)\leq p < 1$.

\begin{lema} If $p\in Spec(Vir(A))$ then $t_A(p)$ is a max - regular element.

\end{lema}
-
\begin{proof} It suffices to prove that $e\in B(A)$ and $e\not\leq p$ imply $t_A(p)\lor e = 1$. Assume by absurdum that there exists $e\in B(A)$ such that $e\not\leq p$ and $t_A(p)\lor e < 1$. The element $t_A(p)\lor e$ is regular so there exists a max - regular element $q$ such that $t_A(p)\lor e < q$. Since the complemented elements $e$ and $\neg e$ are pure, it follows that $e\not\leq p$ and $e\land \neg e = 0$ implies $\neg e\leq p$, hence $\neg e = t_A(\neg e)\leq t_A(p)$. Thus one gets $1\leq e\lor \neg e\leq e\lor t_A(p)$, contradicting $q\in Sp(A)$. We conclude that $t_A(p)\lor e = 1$.

\end{proof}

In accordance with the previous lemma, one can consider the function $t_A: Spec(Vir(A))\rightarrow Sp(A)$ defined by the assignment $p\mapsto t_A(p)$.

\begin{propozitie} The function $t_A$ is surjective and continuous.

\end{propozitie}

\begin{proof} If $q\in Sp(A)$ then there exists $m\in Max(A)$ with $q\leq m$. By Lemma 5.1 $q$ is pure, so $q\leq Vir(m)$. Since $q$ is regular, we have $q = t_A(q)\leq t_A(Vir(m))$. Thus $q = t_A(Vir(m))$, because  both $q$ and $t_A(Vir(m))$ are max - regular. In a straightforward manner one can show that $t_A$ is continous.

\end{proof}

 Following \cite{PasekaRN}, a quantale $A$ is said to be {\emph{normal}} if for all $a,b\in A$ such that $a\lor b = 1$ there exist $e,f\in A$ such that $a\lor e$ = $b\lor f$ = $1$ and $ef = 0$. If $1\in K(A)$ then $A$ is normal if and only if for all $c,d\in K(A)$ such that $c\lor d = 1$ there exist $e,f\in K(A)$ such that $c\lor e$ = $d\lor f$ = $1$ and $ef = 0$ (cf. Lemma 20 of \cite{Cheptea1}). One observes that a commutative ring $R$ is a Gelfand ring iff $Id(R)$ is a normal quantale and a bounded distributive lattice $L$ is normal iff $Id(L)$ is a normal frame.

The normal quantales offer an abstract framework in order to unify some algebraic and  topological properties of commutative Gelfand rings \cite{Johnstone}, \cite{Banaschewski}, \cite{c}, normal lattices \cite{Johnstone}, \cite{GeorgescuVoiculescu}, \cite{Pawar}, \cite{Simmons}, commutative unital $l$ - groups \cite{Birkhoff}, $F$ - rings \cite{Birkhoff}, \cite{Johnstone}, $MV$ - algebras and $BL$ - algebras \cite{Galatos}, \cite{g}, Gelfand residuated lattices \cite{GCM}, etc.

Let us fix a coherent quantale $A$.

\begin{lema}
\cite{GeorgescuVoiculescu2}
Let $A$ be a normal coherent quantale and $a\in A, m\in Max(A)$. Then the following hold:
\usecounter{nr}
\begin{list}{(\arabic{nr})}{\usecounter{nr}}
\item $m$ is the unique maximal element of $A$ such that $O(m)\leq m$;

\item $Ker(m)\leq a$ if and only if $a = m$ or $a = 1$;

\item $Ker(a)\leq m$ implies $a\leq m$;

\item $Vir(a) = Ker(a)$;

\item $Vir(a)\leq m$ if and only if $a\leq m$.
\end{list}
\end{lema}

By Lemmas 4.2,(1) and 5.4,(4), for each maximal element $m$ of $A$ we have $O(m) = Ker(m) = Vir(m)$, so $O(m)$ is pure.

\begin{propozitie}
\cite{Cheptea1}
The quantale $A$ is normal if and only if the reticulation $L(A)$ is a normal lattice (in the sense of \cite{Simmons},\cite{Johnstone}).
\end{propozitie}

For any element $a\in A$ we denote $r(a) = \bigwedge(Max(A)\bigcap V(a))$. In particular, $r(0)$ is exactly the radical $r(A)$ of the quantale $A$ (cf.\cite{GG}). We observe that $r(a) = r(0)$ is an abstractization of the Jacobson radical of a ring.

The literature of ring theory contains several properties that characterize the (commutative) Gelfand rings (see\cite{Simmons},\cite{Johnstone},\cite{c},\cite{Aghajani},\cite{Tar3}). The following result extends the main characterization theorems of Gelfand rings. It collects various conditions that characterize normal quantales. In particular, the below properties (1) - (7) correspond to some conditions from Theorem 4.3 of \cite{Aghajani} and the properties (8) - (14) generalize the conditions contained in Theorem 4.6 of \cite{Tar3}.

\begin{propozitie}
\cite{PasekaRN},\cite{GeorgescuVoiculescu2},\cite{SimmonsC}
If $A$ is a coherent quantale then the following are equivalent:
\usecounter{nr}
\begin{list}{(\arabic{nr})}{\usecounter{nr}}
\item $A$ is a normal quantale;
\item For all distinct $m,n\in Max(A)$ there exist $c_1,c_2\in K(A)$ such that  $c_1\not\leq m$, $c_2\not\leq n$  and $c_1c_2 = 0$;
\item The inclusion $Max(A)\subseteq Spec(A)$ is Hausdorff embedding (i.e. any pair of distinct points in $Max(A)$ have disjoint neighbourhoods in $Spec_Z(A))$;
\item For any $p\in Spec(A)$ there exists a unique $m\in Max(A)$ such that $p\leq m$;
\item $Spec_Z(A)$ is a normal space;
\item The inclusion $Max_Z(A)\subseteq Spec_Z(A)$ has a continuous retraction $\gamma:Spec_Z(A)\rightarrow Max_Z(A)$;
\item If $m\in Max(A)$ then $\Lambda(m)$ is a closed subset of $Spec_Z(A)$.
\item If $m,n$ are two distinct maximal elements of $A$ then $Vir(m)\lor Vir(n) = 1$;
\item For all $a,b\in A$, $a\lor b = 1$ implies $Vir(a)\lor Vir(b) = 1$;
\item $Vir: A\rightarrow Vir(A)$ preserves arbitrary joins;
\item For all $a\in A$ and $m\in Max(A)$, $Vir(a)\leq m$ implies $a\leq m$;
\item For all $a\in A$, the following equality holds: $Max(A)\bigcap V(a)$ = $Max(A)\bigcap V(Vir(a))$;
\item For all $a\in A$, the following equality holds: $r(a) = r(Vir(a))$;
\item The function $\eta: Max_Z(A)\rightarrow Spec_Z(Vir(A))$ defined by $m\mapsto Vir(a)$ is a homeomorphism.
\end{list}

\end{propozitie}

\begin{proof}

$(1)\Leftrightarrow(2)\Leftrightarrow(3)\Leftrightarrow(4)$ By Proposition 3.2 of \cite{GeorgescuVoiculescu2}.

$(1)\Leftrightarrow(8)\Leftrightarrow(9)\Leftrightarrow(10)$ By Proposition 3.4 of \cite{GeorgescuVoiculescu2} or Theorem 3.5 of \cite{SimmonsC}.

$(1)\Leftrightarrow (11)$ By Theorem 3.5 of \cite{SimmonsC}.

$(11)\Leftrightarrow(12)\Leftrightarrow(13)$ These equivalences follow in a straightforward manner.

$(1)\Leftrightarrow(4)\Leftrightarrow(5)$ These equivalences follow from \cite{PasekaRN},\cite{Sun} or by using Proposition 5.5 and \cite{Johnstone}, (p.68, Proposition 3.7).

$(1)\Rightarrow (14)$ By Theorem 3.5 of \cite{GeorgescuVoiculescu2}.

$(14)\Rightarrow (6)$ According to Lemma 4.2,(4), $Vir: Spec_Z(A) \rightarrow Spec_Z(Vir(A))$ is a continuous map and, by the hypothesis (14), the function $\eta: Max_Z(A)\rightarrow Spec_Z(Vir(A))$ defined by $m\mapsto Vir(a)$ is a homeomorphism. Then it is easy to see that $\gamma = \eta^{-1}\circ Vir: Max_Z(A)\rightarrow Spec_Z(Vir(A))$ is a continuous retraction of the inclusion $Max_Z(A)\subseteq Spec_Z(A)$.

$(1)\Leftrightarrow (7)$ According to \cite{Hochster}, for each bounded distributive lattice $L$ there exists a commutative ring $R$ such that the reticulation $L(R)$ of $R$ is isomorphic to $L$. Then for any coherent quantale $A$ there exists a commutative ring $R$ such that the reticulations $L(A)$ and $L(R)$ are isomorphic. By Proposition 5.5, A is a normal quantale iff $L(A)$ is a normal lattice iff $L(R)$ is a normal quantale iff $R$ is a Gelfand ring. Applying twice
Lemma 3.5 we get a homeomorphism $\theta:Spec_Z(A)\rightarrow Spec_Z(R)$; moreover, $\theta$ is an order - isomorphism. Thus the following assertions are equivalent:

$\bullet$ for all $m\in Max(A)$, $\Lambda(m) = \{p\in Spec(A)| p\leq m\}$ is closed in $Spec_Z(A)$;

$\bullet$ for all $M\in Spec(R)$, $\Lambda(M) = \{P\in Spec(R)| P\subseteq M\}$ is closed in $Spec_Z(R)$.

In accordance with the equivalence $(i)\Leftrightarrow (viii)$ from Theorem 4.3 of \cite{Aghajani}, the following equivalences hold: $A$ is a normal quantale iff $R$ is a Gelfand ring iff for all $M\in Max(R)$, $\Lambda(M)$ is closed in $Spec_Z(R)$ iff for all $m\in Max(A)$, $\Lambda(m)$ is closed in $Spec_Z(A)$.

\end{proof}

Let $L$ be a bounded distributive lattice and $P\in Spec_{Id}(L)$, hence $L-P$ is a filter of $L$. Following \cite{Brezuleanu}, the quotient bounded distributive lattice $L_P$ = $L/L-P$ is called the {\emph{lattice of fractions}} of $L$ associated with the prime ideal $P$. The congruence $\equiv_P$ of $L$ modulo the filter $L-P$ has the following form: $x\equiv_P y$ iff $x\land t = y\land t$, for some $t\notin P$. We denote by $x_P$ the congruence class of the element $x\in L$. Let us consider  the lattice morphism $\pi_P: L\rightarrow L_P$, defined by $\pi_P(x) = x_P$, for all $x\in L$.

\begin{lema}
For all $x\in L$, we have $x\in O(P)$ if and only if $\pi_P(x) = 0_P$.
\end{lema}

\begin{proof}For all $x\in L$, the following equivalences hold: $\pi_P(x) = 0_P$ iff $x\equiv_P 0$ iff there exists $t\notin P$ such that $x\land t = 0$ iff $x\in O(P)$.

\end{proof}

\begin{remarca}\cite{Brezuleanu}
For any prime ideal $P$ of $L$, we shall denote $\Lambda_{Id}(P) = \{Q\in Spec_{Id}(L)|Q\subseteq P\}$. Let us consider the function $\pi_P^*: \Lambda_{Id}(L)\rightarrow Spec_{Id}(L)$, defined by the assigment $Q\mapsto \pi_P^*(Q) = \{x_P| x\in Q\}$. It is straightforward to prove that $\pi_P^*$ is an order - isomorphism.

\end{remarca}

\begin{propozitie} For all prime ideal $P$ of the lattice $L$, we have $O(P)$ = $\bigcap \Lambda_{Id}(P)$.

\end{propozitie}

\begin{proof} For any element $x\in O(P)$ there exists $y\notin P$ such that $x\land y = 0$. Then for any $Q\in V_{Id}(L)$ we have $y\notin Q$, hence $x\in Q$. It follows that the inclusion $O(P)\subseteq \bigcap V_{Id}(P)$ holds.

In order to establish the converse inclusion $\bigcap \Lambda_{Id}(P)\subseteq O(P)$, assume that $x\in \bigcap \Lambda_{Id}(P)$. Then $\pi_P(x)\in  \pi_P^*(Q)$ for all prime ideal $Q$ such that $Q\subseteq P$, hence by using Remark 5.8,  one gets

$\pi_P(x)\in \bigcap\{\pi_P^*(Q)|Q\in \Lambda_{Id}(P)\}$ = $\bigcap Spec_{id}(L)$ = $\{0_P\}$.

In accordance with Lemma 5.7, one obtains $x\in O(P)$.

\end{proof}

\begin{teorema}
Let $A$ be a semiprime coherent quantale. For any $m$ - prime element $p$ of $A$, the equality $O(p)$ = $\bigwedge\Lambda(p)$ holds.
\end{teorema}

\begin{proof}
Let $c\in K(A)$ such that $c\leq O(p)$, hence $c\leq p$ and $c^{\perp}\not\leq p$ (cf. Lemma 4.17,(2)). For all $q\in \Lambda(p)$ we have $c^{\perp}\not\leq q$, so $c\leq q$. It follows that $c\leq \bigwedge \Lambda(p)$, hence $O(p)\leq \bigwedge \Lambda(p)$.

In order to prove that $\bigwedge \Lambda(p)\leq O(p)$, let us consider an element $c\in K(A)$ such that $c\leq \bigwedge \Lambda(p)$. Then for all $m$ - prime elements $q$ such that $q\leq p$ we have $c\leq q$. According to Lemma 3.5, it follows that for all prime ideals of $L(A)$ such that $Q\subseteq p^{\ast}$ we have $\lambda_A(c)\in Q$. In accordance with Proposition  5.9 it follows that $\lambda_A(c)\in O(p^{\ast})$. Since the quantale $A$ is semiprime, by using Corollary  4.20 one gets $\lambda_A(c)\in (O(p))^{\ast}$. Therefore, by applying Lemma 3.3 and Corollary 4.20, the following hold: $c\leq ((O(p))^{\ast})_{\ast} = (O(p^{\ast}))_{\ast} = O((p^{\ast})_{\ast}) = O(p)$. We conclude that $\bigwedge \Lambda(p)\leq O(p)$.

\end{proof}

We observe that the previous theorem is obtained by transferring Proposition 5.9 from lattices to quantales by using the reticulation.

\begin{corolar}
Let $A$ be a semiprime coherent quantale. For any minimal $m$ - prime element $p$ of $A$, the equality $O(p)$ = $p$ holds.
\end{corolar}

\begin{corolar} If $A$ is a semiprime coherent quantale then $\bigwedge\{O(m)|m\in Max(A)\} = 0$ .

\end{corolar}

\begin{proof} By using Theorem 5.10, for each $m\in Max(A)$ we have $O(m) = \Lambda(m)$. By observing that $Spec(A) = \bigcup\{\Lambda(m)|m\in Max(A)$, the following equality holds:

$\bigwedge\{O(m)|m\in Max(A)\} = \bigwedge Spec(A) = 0$.

\end{proof}

\begin{lema}
If $a$ is a pure element of a normal coherent quantale $A$ then $a = \bigwedge \{O(m)| m\in Max(A)\bigcap V(a)\}$.
\end{lema}

\begin{proof}
Let $a$ be a pure element of $A$. Recall from Lemma 4.2,(ii) the equality $a = \bigwedge \{Vir(m)| m\in Max(A)\bigcap V(a)\}$. By Lemmas 4.2,(1) and 5.4,(4) one gets $O(m) = Ker(m) = Vir(m)$, hence $a = \bigwedge \{O(m)| m\in Max(A)\bigcap V(a)\}$.

\end{proof}

\begin{lema}
Let $A$ be a normal coherent quantale.
\usecounter{nr}
\begin{list}{(\arabic{nr})}{\usecounter{nr}}
\item If $m\in Max(A)$ then $O(m)$ is pure;
\item If $p$ is a pure $m$ - prime element of $A$ then $O(p) = p$.
\end{list}

\end{lema}

\begin{proof}
(1) By Lemmas 4.2,(1) and 5.4,(4) we have $O(m) = Ker(m) = Vir(m)$;

(2) The inequality $O(p)\leq p$ is always true. In order to prove the converse inequality $p\leq O(p)$ let us consider a compact element $c$ such that $c\leq p$. Since $p$ is pure we have $p\lor c^{\perp} = 1$, so $c^{\perp}\not\leq p$. From $cc^{\perp} = 0$ and $c^{\perp}\not\leq p$ it follows that $c\leq p$, so $p\leq O(p)$.

\end{proof}

\begin{corolar} If $A$ be a normal coherent quantale then a maximal element of $A$ is pure if and only if $O(m) = m$.

\end{corolar}

Following \cite{Cheptea1},\cite{GG}, a quantale $A$ is said to be {\emph{hyperarchimedean}} if for all $c\in K(A)$ there exists a positive integer $n$ such that $c^n\in B(A)$. In \cite{Cheptea1} we proved that a coherent quantale is hyperarchimedean iff $L(A)$ is a Boolean algebra iff $Max(A) = Spec(A)$.

\begin{propozitie}
Let $A$ be a semiprime coherent quantale. Then rhe following are equivalent:
\usecounter{nr}
\begin{list}{(\arabic{nr})}{\usecounter{nr}}
\item $A$ is hyperarchimedean;
\item Any maximal element of $A$ is pure.
\end{list}

\end{propozitie}

\begin{proof}

$(1)\Rightarrow (2)$ Let $m$ be a maximal element of $A$. According to Theorem 1 of \cite{Cheptea1}, we have $Max(A) = Spec(A)$, so $m\in Min(A)$. By Corollary 5.11 we get $O(m) = m$, so m is pure.

$(2)\Rightarrow (1)$ Assume that $m\in Max(A)$, so $m$ is pure, hence $O(m) = m$ (cf. Lemma 5.14). By Proposition 5.10, it folows that $\bigwedge\Lambda(m) = O(m) = m$, hence $m$ is a minimal $m$ - prime element. Therefore $Min(A) = Max(A)$, so, by Theorem 1 of \cite{Cheptea1}, we conclude that $A$ is hyperarchimedean.

\end{proof}

\begin{propozitie} Let $A$ be a normal coherent quantale and $a\in A$. Then $K(a) = \bigwedge(Max(A)\bigcap V(a))$ is a pure element of $A$.

\end{propozitie}

\begin{proof}
Let $c$ be a compact element of $A$ such that $c\leq K(a)$. We have to prove that $K(a)\lor c^{\perp} = 1$. For all $m\in Max(A)\bigcap V(a)$ we have $c\leq O(m)$, hence, by Lemma 4.17,(2), we get $c^{\perp}\not\leq m$, so $d_mc = 0$ and $d_m\not\leq m$ for some $d_m\in Max(A)$.

For all maximal element $m$ such that $m\notin V(a)$ we take a compact element $e_m$ such that $e_m\leq a$ and $e_m\not\leq m$. Assume that

$\bigvee\{d_m|m\in Max(A), a\leq m\}\lor \bigvee\{e_m|m\in Max(A), a\not\leq m\} < 1$,

so there exists $n\in Max(A)$ such that

$\bigvee\{d_m|m\in Max(A), a\leq m\}\lor \bigvee\{e_m|m\in Max(A), a\not\leq m\}\leq n$.

Thus $a\leq n$ implies $d_n\leq n$ and $a\not\leq n$ implies $e_n\leq n$. In the both cases we have obtained a contradiction, therefore

$\bigvee\{d_m|m\in Max(A), a\leq m\}\lor \bigvee\{e_m|m\in Max(A), a\not\leq m\} = 1$.

Since $1\in K(A)$ there exist the maximal elements $m_1,...,m_k, n_1,...,n_l$ such that $d_{m_1}\lor ...\lor  d_{m_k}\lor e_{n_1}\lor... \lor e_{n_l} = 1$, $a\leq m_i$, for $i = 1,...,k$ and $a\not\leq n_j$, for $j = 1,...,l$. Let us denote $d_i = d_{m_i}$ for $i = 1,...,k$ and $e_j = e_{n_j}$ for $j = 1,...,l$.  Thus $cd_i = 0$, $d_i\leq m_i$, for $i = 1,...,k$ and $e_j\leq a$, $e_j\not\leq n_j$, for $j = 1,...,l$.

If $d = d_1\lor...\lor d_k$ and $e = e_1\lor...\lor e_l$ then $d,e\in K(A)$ and $d\lor e = 1$. It follows that $cd = 0$ and $e\leq a$. Since $A$ is a normal quantale, from  $d\lor e = 1$ we infer that there exist $x,y\in K(A)$ such that $d\lor x = e\lor y = 1$ and $xy =0$. Then $x = x(e\lor y) = xe\lor xy = xe$, hence $x\leq e\leq a$. In a similar way we get $y = yd\leq d$.

We shall prove that $x\leq K(a)$. Let $m$ be a maximal element of $A$ such that $a\leq  m$, hence $e\leq a\leq m$. Since $e\lor y = 1$ and $e\leq m$ it follows that $y\not\leq m$. On the other hand, $xy = 0$ implies $y\leq x^{\perp}$, hence $x^{\perp}\not\leq m$. According to Lemma 4.17,(2) we have $x\leq O(m)$. We have proven that for all maximal element $m$, $a\leq m$ implies $x\leq O(m)$, hence $x\leq K(a)$.

 Recall that $cd = 0$, so $d\leq c^{\perp}$. Thus $1 = x\lor d\leq K(a)\lor c^{\perp}$, so $K(a)\lor c^{\perp} = 1$. We conclude that $K(a)$ is pure.

\end{proof}

\begin{teorema}
Let $A$ be a normal coherent quantale. Then an element $a\in A$ is pure if and only if $a = \bigwedge\{O(m)|m\in Max(A)\bigcap E\}$, for some closed subset $E$ of $Spec_Z(A)$.

\end{teorema}

\begin{proof}
By Lemma 5.13, any pure element $a\in A$ has the form $a = \bigwedge\{O(m)|m\in Max(A)\bigcap V(a)\}$. Conversely, assume that $a = \bigwedge\{O(m)|m\in Max(A)\bigcap E\}$, for some closed subset $E$ of $Spec_Z(A)$. Then $E = V(x)$ for some $x\in A$, hence $a = K(x)$. In accordance with Proposition 5.17, $a = K(x)$ is a pure element of $A$.

\end{proof}

For any $m$ - prime element $p$ of a coherent quantale denote $\Omega(p) = \bigwedge \Lambda(p)$. According to Theorem 5.10, if $A$ is semiprime and $p\in Spec(A)$ then $O(p) = \Omega(p)$.

\begin{corolar}
Let $A$ be a normal and semiprime coherent quantale. Then an element $a\in A$ is pure if and only if $a = \bigwedge\{\Omega(m)|m\in Max(A)\bigcap E\}$, for some closed subset $E$ of $Spec_Z(A)$.

\end{corolar}

\section{Pure elements in $PF$ - quantales}

 Following \cite{GG}, a quantale $A$ is said to be an {\emph{$mp$ - quantale}} if for any $p\in Spec(A)$ there exists a unique $q\in Min(A)$ such that $q\leq p$. An $mp$ - frame is an $mp$ - quantale wich is a frame. We remark that a ring $R$ is an {\emph{$mp$ - ring}} in the sense of \cite{Aghajani} if and only if the quantale $Id(R)$ of ideals of $R$ is an $mp$ - quantale. The $mp$ - quantales can be related to the conormal lattices, introduced by Cornish in \cite{Cornish} under the name of "normal lattices". According to \cite{Simmons},\cite{Johnstone}, a {\emph{conormal lattice}} is a bounded distributive lattice $L$ such that for all $x,y\in L$ with $x\land y = 0$ there exist $u,v\in L$ having the properties $x\land u$ = $y\land v$ = $0$ and $u\lor v$ = $1$. Cornish obtained in \cite{Cornish} several characterizations of the conormal lattices. For example, a bounded distributive lattice $L$ is conormal if and only if any prime ideal of $L$ contains a unique minimal prime ideal.

\begin{lema}
\cite{GG}
A coherent quantale $A$ is an $mp$ - quantale if and only if the reticulation $L(A)$ is a conormal lattice.
\end{lema}

Let us denote by $Min_Z(A)$ (resp. $Min_F(A)$) the topological space obtained by restricting the topology of $Spec_Z(A)$ (resp. $Spec_F(A)$) to $Min(A)$. By using Lemma 3.5, $Min_Z(A)$ is homeomorphic to the space $Min_{Id,Z}(A)$ of minimal prime ideals in $L(A)$ with the Stone topology and $Min_F(A)$ is homeomorphic to the space $Min_{Id,F}(A)$ of minimal prime ideals in $L(A)$ with the flat topology. Then $Min_Z(A)$ is a zero - dimensional Hausdorff space and $Min_F(A)$ is a compact $T1$  space \cite{GG}.

\begin{teorema}\cite{GG}
If $A$ is a semiprime quantale then the following are equivalent:
\usecounter{nr}
\begin{list}{(\arabic{nr})}{\usecounter{nr}}

\item $Min_Z(A)$ = $Min_F(A)$;

\item $Min_Z(A)$ is a compact space;

\item $Min_Z(A)$ is a Boolean space;

\item For any $c\in K(A)$ there exists $d\in K(A)$ such that $cd = 0$ and $(c\lor d)^{\perp}$ = $0$.

\end{list}

\end{teorema}

\begin{teorema}\cite{GG}
If $A$ is a coherent quantale then the following are equivalent:
\usecounter{nr}
\begin{list}{(\arabic{nr})}{\usecounter{nr}}

\item $A$ is an $mp$ - quantale;

\item For any distinct elements $p,q\in Min(A)$ we have $p\lor q = 1$;

\item $R(A)$ is an $mp$ - frame;

\item $[\rho(0))_A$ is an $mp$ - quantale;

\item The inclusion $Min_F(A)\subseteq Spec_F(A)$ has a flat continuous retraction;

\item $Spec_F(A)$ is a normal space;

\item If $p\in Min(A)$ then $V(p)$ is a closed subset of $Spec_F(A)$.

\end{list}

\end{teorema}

Recall from \cite{Al-Ezeh2} that a commutative ring $R$ is said to be {\emph{$PF$ - ring}} if the annihilator of each element of $R$ is a pure ideal. Following \cite{GG}, a quantale $A$ is a {\emph{$PF$ - quantale}} if for each $c\in K(A)$, $c^{\perp}$ is a pure element. For any commutative ring $R$, $Id(R)$ is a $PF$ - quantale if and only if $R$ is a $PF$ - ring.

\begin{lema}\cite{GG}
\usecounter{nr}
\begin{list}{(\arabic{nr})}{\usecounter{nr}}
\item Any $PF$ - quantale $A$ is semiprime;
\item If $A$ is a $PF$ - quantale then the reticulation $L(A)$ is a conormal lattice.
\end{list}
\end{lema}

Now we remind from \cite{GG} some characterization theorems of $PF$ - quantales. In what follows they will be intensely used in proving some algebraic and topological results on pure elements in a $PF$ - algebra.

\begin{teorema}\cite{GG}
Let $A$ be a coherent quantale. Then $A$ is a $PF$ - quantale if and only if  $A$ is a semiprime $mp$ - quantale.

\end{teorema}

\begin{teorema} \cite{GG} For a coherent quantale $A$ consider the following conditions:
\usecounter{nr}
\begin{list}{(\arabic{nr})}{\usecounter{nr}}
\item Any minimal $m$ - prime element of $A$ is pure;

\item $A$ is an $mp$ - quantale.
\end{list}
Then $(1)$ implies $(2)$. If the quantale $A$ is semiprime then the converse implication holds.
\end{teorema}

\begin{corolar}\cite{GG} Let $A$ be a semiprime quantale. Then $A$ is a $PF$ - quantale if and only if any minimal $m$ - prime element of $A$ is pure.

\end{corolar}

\begin{teorema}\cite{GG}
For a coherent quantale $A$ the following are equivalent:
\usecounter{nr}
\begin{list}{(\arabic{nr})}{\usecounter{nr}}

\item $A$ is a $PF$ - quantale;

\item $A$ is a semiprime $mp$ - quantale;

\item If $c,d\in K(A)$ then $cd = 0$ implies $c^{\perp}\lor d^{\perp} = 1$;

\item If $c,d\in K(A)$ then $(cd)^{\perp}$ = $c^{\perp}\lor d^{\perp}$;

\item For each $c\in K(A)$, $c^{\perp}$ is a pure element.

\end{list}
\end{teorema}

\begin{propozitie}
Let $A$ be a coherent $PF$ - quantale. If $p\in Spec(A)$ then $O(p)\in Min(A)$.
\end{propozitie}
\begin{proof}
By taking into account Theorem 5.10, it suffices to show that $O(p)$ is an $m$ - prime element of $A$. Let us consider two compact elements $c,d$ of $A$ such that $cd\leq O(p)$. By Lemma 4.17,(2) it follows that  $(cd)^{\perp}\not\leq p$. Since $(cd)^{\perp}$ = $c^{\perp}\lor d^{\perp}$ (by Theorem 6.8,(3)), we have $c^{\perp}\lor d^{\perp}\not\leq p$, hence $c^{\perp}\not\leq p$ or $d^{\perp}\not\leq p$. According to Lemma 4.17,(2), we get $c\leq O(p)$ or $d\leq O(p)$, hence $O(p)$ is $m$ - prime.

\end{proof}

\begin{teorema}
Let $A$ be a $PF$ - quantale. The pure elements of $A$ have the form $\bigwedge (Min(A)\bigcap E)$, where $E$ is a closed subset of $Spec_F(A)$.
\end{teorema}

\begin{proof}
 Let $a$ be a pure element of $A$.  According to Theorem 6.6, the pure element $a$ is minimal $m$ - prime, hence $a = \bigwedge (Min(A)\bigcap V(a))$. By Theorem 6.3,(7), $Min(A)\bigcap V(a)$ is a closed subset of $Min_F(A)$.

Conversely, assume that $E$ is a closed subset of $Spec_F(A)$ and $a = \bigwedge (Min(A)\bigcap E)$. Let $c$ be a compact element of $A$ such that $c\leq a$. We have to prove that $a\lor c^{\perp} = 1$. Assume by absurdum that $a\lor c^{\perp} < 1$, so $a\lor c^{\perp}\leq m$, for some $m\in Max(A)$. Let us consider a minimal $m$ - prime element $q$ of $A$ such that $q\leq m$.

Assume that $q\in E$, so we have $c\leq q$. By Theorem 6.6, the minimal $m$ - prime element $q$ is pure, hence $q\lor c^{\perp} = 1$. From $q\leq m$ and $c^{\perp}\leq m$ we obtain $1 = q\lor c^{\perp}\leq m$, contradicting that $m$ is a maximal element. We conclude that $q\notin E$, so $q\neq p$ for all $p\in Min(A)\bigcap E$.

Since $A$ is an $mp$ - quantale, we can apply Theorem 6.3,(2), hence for all $p\in Min(A)\bigcap E$ there exist $s_p, y_p\in K(A)$ such that $x_p\leq p$, $y_p\leq q$ and $x_p\lor y_p = 1$. Assume $r\in E$ and take $p\in Min(A)$ such that $p\leq r$. It follows that $x_p\leq p\leq r$, so $r\in V(x_p)$. Thus we obtain the following inclusion:

$E\subseteq \bigcup \{V(x_p)| p\in Min(A)\bigcap E\}$.

We remind that $E$ is a closed subset of the compact space $Spec_F(A)$, hence $E$ is itself compact. Then there exist a positive integer $n$ and $p_1,...,p_n\in Min(A)\bigcap E$ such that $E\subseteq V(x_{p_1})\bigcup...\bigcup V(x_{p_n})$.

Let us denote $x_i = x_{p_i}$, $y_i = y_{p_i}$  for all $i = 1,...,n$ and $x = x_1\cdot...\cdot x_n$, $y = y_1\lor...\lor y_n$. Thus $x,y\in K(A)$ and $y\lor x_i = 1$ for all $i = 1,...,n$, hence by applying Lemma 2.1,(3) we obtain $y\lor x = y\lor x_1\cdot...\cdot x_n = 1$.

Since $y_i\leq q$ for all $i = 1,...,n$, we get $y\leq q$. For all $r\in Min(A)\bigcap E$ we have $r\in V(x_1)\bigcup...\bigcup V(x_n)$, so there exists $i\in \{1,...,n\}$ such that $x_i\leq r$. It follows that $x\leq x_i\leq r$. Thus for all  $r\in Min(A)\bigcap E$ we have $x\leq r$, hence $x\leq \bigwedge (Min(A)\bigcap E) = a$.

From $x\leq a\leq m$ and $y\leq q\leq m$ we obtain $1 = x\lor y\leq m$, contradicting $m\in Max(A)$. Therefore $a\lor c^{\perp} = 1$, so $a$ is a pure element of $A$.

\end{proof}

Let $R$ be a $PF$ - ring. If we apply Theorem 6.10 for the $PF$ - quantale $A = Id(R)$ then we obtain Theorem 7.3 of \cite{Aghajani}.

\begin{lema}
Let $A$ be a coherent quantale. For any $m\in Max(Vir(A))$ there exists $n\in Max(A)$ such that $Vir(n) = m$.
\end{lema}

\begin{proof}
Let $m$ be an element of $Max(Vir(A))$ and $n\in Max(A)$ such that $m\leq n$. By Lemma 3.2,(4) one can consider the function $Vir: Spec(A)\rightarrow Spec(Vir(A))$. Then $m = Vir(m)\leq Vir(n)$, so we get $m = Vir(n)$, because $m\in Max(Vir(A))$ and $Vir(n)\in Spec(Vir(A))$.
\end{proof}

\begin{teorema}
If $A$ is a coherent $PF$ - quantale then $Min(A) = Max(Vir(A))$.
\end{teorema}

\begin{proof} Firstly we shall establish the inclusion $Min(A)\subseteq Max(Vir(A))$. Assume that $p$ is a minimal $m$ - prime element of $A$. By Corollary 6.7, $p$ is a pure element of $A$.

Since $p < 1$ there exists $m\in Max(Vir(A))$ such that $p\leq m$. Let us consider a compact element $x$ of $A$ such that $x\leq m$. Since $m$ is pure we have $m\lor x^{\perp} = 1$, hence $x^{\perp}\not\leq m$. Since $p\leq m$, we get $x^{\perp}\not\leq m$. It follows that $x\leq m$, hence $p = m$. We conclude that $p\in Max(Vir(A))$.

Conversely, assume that $m\in Max(Vir(A))$. By the previous lemma there exists $n\in Max(A)$ such that $Vir(n) = m$. Let us take a minimal $m$ - prime element $p$ such that $p\leq n$. According to the first part of the proof, $p$ is a maximal element of the frame $Vir(A)$. Thus $p = Vir(p)\leq Vir(n) = m$, hence $p  = m$, because $m,p\in Max(Vir(A))$. It results that $m\in Min(A)$, hence the inclusion  $Max(Vir(A))\subseteq Min(A)$ is established.

\end{proof}

\begin{propozitie} Let $A$ be a coherent $PF$ - algebra and $a\in A$. Then $K(a) = \bigwedge(Max(A)\bigcap V(a))$ is a pure element of $A$.

\end{propozitie}

\begin{proof}
Let $c$ be a compact element of $A$ such that $c\leq K(a)$. We shall prove that $K(a)\lor c^{\perp} = 1$. By using the same argument as in the proof of Theorem  one can find two compact element $d$ and $e$ such that $e\leq a$, $cd = 0$ and $d\lor e = 1$. Since $A$ is a $PF$ - quantale, from $cd = 0$ we get $c^{\perp}\lor d^{\perp} = 1$ (cf. Theorem 6.8,(4)). Then there exist the compact elements $x$ and $y$ such that $x\leq c^{\perp}$, $y\leq d^{\perp}$ and $x\lor y = 1$.

We shall prove that $y\leq K(a)$. Let $m$ be a maximal element of $A$ such that $a\leq m$, hence $e\leq a\leq m$. If $d\leq m$ then $1 = e\lor d\leq m$, contradicting $m\in Max(A)$. Then we obtain $d\not\leq m$. From $y\leq d^{\perp}$ we get $d\leq y^{\perp}$, hence $y^{\perp}\not\leq m$. According to Lemma 4.17,(2) we have $y\leq O(m)$. We have proven that for all maximal element $m$, $a\leq m$ implies $y\leq O(m)$, hence $y\leq K(a)$.

We remark that $1 = y\lor x\leq K(a)\lor c^{\perp}$, so $K(a)\lor c^{\perp} = 1$. We conclude that $K(a)$ is pure.

\end{proof}

\begin{lema} Let $A$ be a coherent $PF$ - quantale, $a\in Var(A)$ and $m\in Max(A)$ such that $a\leq m$. Assume that the following condition holds:

(*) For all $c\in K(A)$, $c\leq O(m)$ implies $K(c)\leq a$.

Then $a = O(m)$.

\end{lema}

\begin{proof} Since $a\leq m$ and $a$ is pure we get $a = Vir(a)\leq Vir(m) = O(m)$. In order to show that $O(m)\leq a$, assume that $c$ is a compact element such that $c\leq O(m)$. Since $O(m) = Vir(m)$ is pure we have $O(m)\lor c^{\perp} = 1$, so there exist the compact elements $d,e$ such that $d\leq O(m)$, $e\leq c^{\perp}$ and $d\lor e = 1$.

Now we shall prove that $c\leq K(d)$. Let us consider a maximal element $n$ of $A$ such that $d\leq n$. Assume that $c^{\perp}\leq n$, so $e\leq c^{\perp}\leq n$, therefore $1 = d\lor e\leq n$. This contradicts $n\in Max(A)$, hence $c^{\perp}\not\leq n$. By Lemma 4.17,(2) we obtain $c\leq O(n)$. It follows that $c\leq\bigwedge\{O(n)| n\in Max(A), d\leq n \} = K(d)$. In accordance with (*), $d\leq O(m)$ implies $K(d)\leq a$, hence $c\leq K(d)\leq a$. It follows that
$O(m) = \bigvee\{c\in K(A)|c\leq O(m)\}\leq a$.

\end{proof}

\begin{teorema} If $A$ is a coherent $PF$ - quantale then $Spec(Vir(A)) = Max(Vir(A))$.

\end{teorema}

\begin{proof} It suffices to prove that $Spec(Vir(A))\subseteq Max(Vir(A))$. Assume that $p\in Spec(Vir(A))$. Let $q$ be a maximal element of the frame $Vir(A)$ such that $p\leq q$.

Let us consider a compact element $c$ of $A$ such that $c\leq q$. For each $n\in Max(A)$ there exist two possible casses:

$Case 1: c\leq n$. By the definition of $K(c)$ we have $K(c)\leq O(n)$.

$Case 2: c\not\leq n$. Since $n$ is $m$ - prime it follows that $c^{\perp}\leq n$. In accordance with Theorem 6.8,(5), $c^{\perp}$ is a pure element of $A$, so $c^{\perp}\leq Vir(n) = O(n)$.

In virtue of these two cases one obtains $K(c)c^{\perp}\leq \bigwedge\{O(n)| n\in Max(A)$, hence $K(c)c^{\perp} = 0$ (by Corrolary 5.12).

Since $c\leq q = O(q)$ we have $c^{\perp}\not\leq q$ (by Lemma 4.17,(2)). According to Theorem 6.8,(5) and Proposition 6.13, $c^{\perp}$ and $K(c)$ are pure elements, hence $K(c)\leq p$.

We have proven that for each compact element $c\leq q$ we have $K(c)\leq p$. We remark that $q < 1$ (because $q\in Max(Vir(A))$). Let us consider $m\in Max(A)$ such that $q\leq m$. Thus $q = Vir(q)\leq Vir(m)$, hence $q = Vir(m) = O(m)$. Therefore for each compact element $c$, $c\leq O(m)$ implies $K(c)\leq p$, i.e. the condition $(*)$ is satisfied (with $p$ instead of $a$). By applying Lemma 6.14 we conclude that $p = O(m) =q$, so $p\in Max(Vir(A))$.

\end{proof}

\begin{remarca} Let $A$ be a $PF$ - quantale. By applying the previous result and Proposition 6.2 of \cite{GG} it follows that $Vir(A)$ is a hyperarchimedean frame.Thus, by using Theorem 6.8,(8) of \cite{GG}, we have the following identity of topological spaces: $Spec_Z(Vir(A)) = Spec_F(Vir(A))$. Taking into account Theorem 6.12, it follows that $Min_F(A)$ and $Spec_Z(Vir(A))$ are identical as topological spaces.

\end{remarca}

\begin{remarca} Let $R$ be a $PF$ - ring. By applying the previous Theorems 6.12 and 6.16 to the $PF$ - quantale $A = Id(R)$ we obtain Theorems 5.1 and 5.3 of \cite{Tar3}. Remark 6.17 can be view as an abstract version of Theorem 5.4 of \cite{Tar3}.

\end{remarca}

\section{Purified quantales}

Let us recall from \cite{Aghajani} that a commutative ring $R$ is a {\emph{purified ring}} if for all distinct minimal ideals $P$ and $Q$ of $R$ there exists an idempotent element $e$ such that $e\in P$ and $1 - e\in Q$. A quantale $A$ is said to be a {\emph{purified quantale}} if for all minimal $m$ - prime elements $p$ and $q$ of $A$ there exists $e\in B(A)$ such that $e\leq p$ and $\neg e\leq q$. Then a commutative ring $R$ is a purified ring if and only in $Id(R)$ is a purified quantale.

The results of this section generalize to purified quantales some theorems proved in \cite{Aghajani} for purified rings.

\begin{lema}
Any purified quantale $A$ is an $mp$ - quantale.
\end{lema}

\begin{proof}
Let $p,q$ be two distinct elements of $Min(A)$. Thus there exists a complemented element $e$ of $A$ such that $e\leq p$ and $\neg e\leq q$, hence $p\lor q = 1$. In accordance to Theorem 6.3,(2) it follows that $A$ is a $mp$ - quantale.
\end{proof}

Let $A$ be a coherent quantale and $a\in A$. Let us consider the quantale $[a)_A$ (defined in Section 2) and the unital quantale morphism $u_a^A : A \rightarrow [a)_A$, defined by $u_a^A(x)=x \lor a$, for all $x \in A$. It is easy to see that the negation operation $\neg^a$ of $[a)_A$ is defined by $\neg^a(x)$ = $\neg x\lor a$, for any $x\in [a)_A$. By Lemma 14 of \cite{Cheptea1}, one can consider the Boolean morphism $B(u_a^A): B(A)\rightarrow B([a)_A)$ defined by $B(u_a^A)(e) = u_a^A(e)$, for all $e\in B(A)$. According to \cite{Cheptea1}, an element $a\in A$ has the lifting property ($LP$) if the Boolean morphism $B(u_a^A): B(A)\rightarrow B([a)_A)$ is surjective. The quantale $A$ has $LP$ if each element of $A$ has $LP$.

\begin{lema}
$\rho(0)$ has $LP$.
\end{lema}

\begin{proof}
Let us denote $a = \rho(0)$. We have to prove that $B(u_a^A))$ is surjective. Let $x$ be a complemented element of $[a)_A$, so there exists $y\in [a)_A$ such that $x\lor y = 1$ and $xy\lor a = x \cdot_a y = a$. Since $1$ is compact, there exist $c,d\in K(A)$ such that $c\leq x$, $d\leq y$ and $c\lor d = 1$. On the other hand, $xy\leq a = \rho(0)$ implies $cd\leq \rho(0)$, so by Lemma 2.4,(2), there exists a positive integer $n$ such that $c^nd^n = 0$. In accordance with Lemma 2.1,(2) we have  $c^n\lor d^n = 1$. By using Lemma 3.8(6), it follows that $c^n, d^n\in B(A)$ and $c^n = \neg d^n$. One remarks that $u_a^A(c^n)\leq x$ and $u_a^A(d^n)\leq y$. The second inequality implies $x = \neg^a y \leq \neg^a(u_a^A(d^n)) = u_a^A(\neg d^n) = u_a^A(c^n)$. We have proven that $u_a^A(c^n) = x$ and $c^n\in B(A)$, hence $a = \rho(0)$ has $LP$.
\end{proof}

\begin{teorema} A coherent quantale $A$ is a purified quantale if and only if $[\rho(0))_A$ is a purified quantale.

\end{teorema}

\begin{proof}
Let us denote $a = \rho(0)$. Assume that $A$ is a purified quantale and consider $p,q\in Min([a)_A)$ such that $p\neq q$. Observing that $Min(A) = Min([a)_A)$, it results that there exists $e\in B(A)$ such that $e\leq p$ and $\neg e\leq q$. It follows that $u_a^A(e)\in B([a)_A)$, $u_a^A(e)\leq u_a^A(p) = p$ and $\neg^au_a^A(e) = u_a^A(\neg e)\leq u_a^A(q) = q$, therefore $[a)_A$ is a purified quantale.

Conversely, suppose that $[a)_A$ is a purified quantale and consider two distinct minimal $m$ - prime elements $p,q$ of $A$. Then $p,q\in Min([a)_A)$, hence, by taking into account the hypothesis that $[a)_A$ is a purified quantale there exists $f\in B([a)_A)$ such that $f\leq p$ and $\neg^a f\leq q$. In accordance with Lemma 7.2, $a = \rho(0)$ has the lifting property, so there exists $e\in B([a)_A)$ such that $a\lor e =  u_a^A(e) = f$. Sine $B( u_a^A)$ is a Boolean morphism, we have $a\lor \neg e  =  u_a^A(\neg e) = \neg^au_a^A(e) = \neg^a f$, so $\neg^a e\leq q$. Thus $A$ is a purified quantale.
\end{proof}

\begin{lema} Let $U$ be a subset of $Spec_F(A)$. Then $U$ is a clopen subset of $Spec_F(A)$ if and only if $U = V(e)$, for some $e\in B(A)$.

\end{lema}

\begin{proof}
Assume that $U$ is a clopen subset of $Spec_F(A)$. Then $V = Spec_F(A) - U$ is a clopen subset of $Spec_F(A)$ and $U\bigcup V = Spec_F(A)$, $U\bigcap V = \emptyset$. Since $Max_F(A)$ is compact, there exist $e,f\in K(A)$ such that $V(e)\subseteq U$, $V(f)\subseteq V$ and $V(ef) = V(e)\bigcup V(f) = Spec_F(A)$, hence $ef\leq \rho(0)$. We remark that $V(e\lor f) = V(e)\bigcap V(f)\subseteq U\bigcap V = \emptyset$, so $e\lor f = 1$. From $ef\leq \rho(0)$ we have $e^nf^n = 0$, for some positive integer $n$. By Lemma 2.1,(2) we have $e^n\lor f^n = 1$, therefore by using Lemma 3.8,(6), we obtain $e^n, f^n\in B(A)$. Now it is easy to prove that $U = V(e^n)$ and $V = V(f^n)$. The converse implication is obvious.

\end{proof}

Recall from \cite{Johnstone}, p.69 that a topological space $X$ is said to be

${\bullet}$ {\emph{totally disconnected}}, if the only connected subsets of $X$ are single points;

${\bullet}$ {\emph{totally separated}}, if for all distinct points $x,y\in X$, there exists a clopen subset of $X$ containing $x$ but not $y$.

\begin{teorema}
For a semiprime coherent quantale $A$ the following are equivalent:
\usecounter{nr}
\begin{list}{(\arabic{nr})}{\usecounter{nr}}

\item $A$ is a purified quantale;

\item $A$ is an $mp$ - quantale and $Min_F(A)$ is totally separated;

\item $A$ is an $mp$ - quantale and $Min_F(A)$ is totally disconnected;

\item $A$ is an $mp$ - quantale and $Min_F(A)$ is a Boolean space;

\item The family $(V(e)\bigcap Min(A))_{e\in B(A)}$ is a basis of open sets for $Min_F(A)$.

\item Any minimal $m$ - prime element $p$ of $A$ is regular;

\item $Min(A) = Sp(A)$;

\item $A$ is an $mp$ - quantale and any pure element of $A$ is regular.

\end{list}
\end{teorema}

\begin{proof}
Recall from Corollary 8.6 of \cite{GG} that $Min_F(A)$ is a compact T1 - space. According to Theorem 4.2 of \cite{Johnstone}, the properties $(2)$,$(3)$ and $4$ are equivalent.

$(1)\Rightarrow(2)$ In accordance with Lemma 7.1, $A$ is an $mp$ - quantale. Then for a distinct points $p,q$ of $Min_F(A)$ there exists $e\in B(A)$ such that $e\leq p$, $\neg e\leq q$, so $p\in V(e)$, $q\in V(\neg e)$ and $V(e)\bigcap V(\neg e) = \emptyset$. Then $Min_F(A)$ is totally separated.

$(4)\Rightarrow(5)$ We apply Lemma 7.4.

$(5)\Rightarrow(1)$ Let $p,q$ be two distinct points of $Min(A)$. Then $U = Spec(A) - \{q\}$ = $Spec(A) - \Lambda(q)$ is an open subset of $Spec_F(A)$ and $p\in U$, hence there exists $e\in B(A)$ such that $p\in V(e)\bigcap Min_F(A)\subseteq U\bigcap Min_F(A)$. It follows that $e\leq p$ and $\neg e\leq q$, so $A$ is a purified quantale.

$(1)\Rightarrow(6)$
Assume that $p\in Min(A)$ and $c\in K(A)$ such that $c\leq p$. The property (1) is equivalent to (4), hence $Min_F(A)$ is a Boolean space. Since $p\in V(c)$ and $V(c)\bigcap Min(A)$ is an open subset of $Min_F(A)$ one can find an element $e\in B(A)$ such that $p\in V(e)\bigcap Min(A)\subseteq V(c)\bigcap Min(A)$. Therefore

$c\leq \bigwedge(V(c)\bigcap Min(A)\leq \bigwedge(V(e)\bigcap Min(A)\leq \rho(e))$,

so $c^n\leq e$ for some positive integer $n$. According to Theorem 6.8,(4) we have $(c^n)^{\perp} = c^{\perp}$, hence $c^n\leq e$ implies $\neg e = e^{\perp} \leq (c^n)^{\perp} = c^{\perp}$. It follows that $c\leq (c^{\perp})^{\perp}\leq \neg\neg e = e$. We have proven that $p\leq \bigvee \{e\in B(A)|e\leq p\}$, hence $p = \bigvee \{e\in B(A)|e\leq p\}$. We conclude that the minimal $m$ - prime element $p$ is regular.

$(6)\Rightarrow(1)$
Let $p,q$ be two distinct minimal $m$ - prime elements of $A$. By taking into account the hypothesis, we have $p = \bigvee \{e\in B(A)|e\leq p\}$ and $q = \bigvee \{e\in B(A)|e\leq q\}$. Since $p,q$ are distinct we can find an element $e\in B(A)$ such that $e\leq p$ and $e\not\leq q$. Then $\neg e\leq q$, hence $A$ is a purified quantale.

$(6)\Rightarrow(7)$
We shall prove that $Min(A)\subseteq Sp(A)$. Consider an element $p\in Min(A)$, so by the condition (6), $p$ is regular. In order to show that $p$ is max - regular, let us consider $e\in B(A)$ such that $e\not\leq p$. Thus $\neg e\leq p$, hence $1 = e\lor \neg e\leq p\lor e$, so $p\lor e = 1$. It follows that $p\in Sp(A)$.

Now we shall prove that $Sp(A)\subseteq Min(A)$. Let us consider $p\in Sp(A)$ and take an element $q\in Spec(A)$ such that $s_A(q) = p$ (because the function $s_A: Spec(A)\rightarrow Sp(A)$ is surjective). If $r$ is a minimal $m$ - prime element of $A$ such that $r\leq q$ then we have $s_A(r)\leq s_A(q) = p$, so $s_A(r) = p$ (because $p$ and $s_A(r)$ are max - regular). The minimal $m$ - prime element $r$ is regular, hence $s_A(r) = r$. Thus $r = p$, hence we get $p\in Min(A)$.

$(7)\Rightarrow(1)$
Obviously.

$(5)\Rightarrow(8)$
We have proven that (5) is equivalent to (2), so $A$ is a semiprime $mp$ - quantale (i.e. a $PF$ - quantale). Let $a$ be a pure element of $A$. By Theorem 6.6, $a$ is a minimal $m$ - prime element, hence $V(a)$ is a closed subset of $Spec_F(A)$. Then $U = Spec(A) - V(a)$ is an open subset of $Spec_F(A)$. Applying the hypothesis (5) we find a family $(e_i)_{i\in I}$ of complemend elements of $A$ such that $U\bigcap Min(A)$ = $\displaystyle \bigcup_{i\in I}(V(e_i)\bigcap Min(A))$.

We shall prove that $U = \displaystyle \bigcup_{i\in I}V(e_i)$. Let $p\in \displaystyle \bigcup_{i\in I}V(e_i)$, so $e_j\leq p$ for some $j\in I$. Assume by absurdum that $a\leq p$, so $a\notin U\bigcap Min(A)$. By taking into account the equality $U\bigcap Min(A)$ = $\displaystyle \bigcup_{i\in I}(V(e_i)\bigcap Min(A))$, it follows that $a\notin V(e_j)$,  i.e. $e_j\not\leq a$. Therefore $\neg e_j\leq a\leq p$, so we obtain $1 = e_j\lor \neg e_j \leq p$, contradicting that $p\in Spec(A)$. It follows that $a\not\leq p$, so $p\in U$.

In order to prove the converse inclusion $U\subseteq \displaystyle \bigcup_{i\in I}V(e_i)$, let us assume that $p\notin \displaystyle \bigcup_{i\in I}V(e_i)$, hence for all $i\in I$ we have $e_i\not\leq p$. Consider a minimal $m$ - prime element $n$ such that $n\leq p$, hence $e_i\not\leq n$ for all $i\in I$. This implies $n\notin V(e_i)\bigcap Min(A)$ for all $i\in I$, hence $n\notin U\bigcap Min(A)$. Since $n\notin U$ implies $a\leq n$, we obtain $a\leq n\leq p$, hence
$p\in U$. We have proven that $U = \displaystyle \bigcup_{i\in I}V(e_i)$, hence the following equalities hold:

$V(a)$ = $Spec(A) - \displaystyle \bigcup_{i\in I}V(e_i)$ = $\displaystyle \bigcap_{i\in I}D(e_i)$  =  $\displaystyle \bigcap_{i\in I}V(\neg e_i)$= $V(\displaystyle \bigvee_{i\in I}\neg e_i)$.

Let us consider the regular element $b = \displaystyle \bigvee_{i\in I}\neg e_i$. Then $\rho(a)$ = $\bigwedge V(a)$ = $\bigwedge V(b)$ = $\rho(b)$. By Lemma 5.1, the regular element $b$ is pure. In accordance with Proposition 4.5, for the pure elements $a$ and $b$ we have $a = \rho(a) = \rho(b) = b$. Therefore $a$ is regular element.

$(8)\Rightarrow(6)$
According to Theorem 6.6, any minimal $m$ - prime element $p$ is pure, so $p$ is regular.

\end{proof}

\begin{corolar}
Any hyperarchimedean coherent quantale $A$ is purified.
\end{corolar}

\begin{proof}
In accordance with the characterization theorems of hyperarchimedean quantales given in \cite{Cheptea1},\cite{GG} it follows that $Max(A) = Spec(A) = Min(A)$ and $Spec_F(A)$ is a Boolean space. Thus $A$ is an $mp$ - quantale (by Theorem 6.3,(6)) and $Min_F(A)$ is a Boolean space. By applying Theorem 7.5,(4) we conclude that $A$ is purified.
\end{proof}

\begin{corolar}
Let $A$ be a coherent $PF$ - quantale. If $Min_Z(A)$ is compact then $A$ is purified.
\end{corolar}

\begin{proof}
Since $Min_Z(A)$ is compact it follows that $Min_Z(A) = Min_F(A)$ (cf.Theorem 6.2,(1)). We know that $Min_Z(A)$ is a zero - dimensional Hausdorff space and $Min_F(A)$ is a compact space, hence in our case, $Min_F(A)$ is a Boolean space. By applying Theorem 7.5,(4) we conclude that $A$ is a purified quantale.
\end{proof}

\section{$PP$ - quantales}

\hspace{0.5cm} In this section we shall define the $PP$ - quantales as an abstraction of $PP$ - rings (= Baer rings) \cite{Aghajani},\cite{Lam},\cite{Simmons}, Stone lattices \cite{BalbesDwinger},\cite{Simmons}, Stone $MV$ - algebras \cite{Belluce}, Baer $BL$ - algebras \cite{g}, Stone residuated lattices \cite{Muresan},\cite{Rasouli},etc.

Let $A$ be an algebraic quantale such that $1\in K(A)$. Then $A$ will be called a {\emph{$PP$ - quantale}} if for any $c\in K(A)$ we have $c^{\perp}\in B(A)$. A {\emph{$PP$ - frame}} is a $PP$ - quantale which is a frame.

Let $L$ be a bounded distributive lattice. Following \cite{BalbesDwinger},\cite{Simmons}, $L$ is said to be a {\emph{Stone lattice}} if for any $x\in L$ there exists $e\in B(L)$ such that $Ann(x)$ is the ideal $[e)$ of $L$ generated by the point set $\{x\}$. Then $L$ is a Stone lattice if and only if $Id(L)$ is a $PP$ - frame.

Let $R$ be a commutative ring. Then $R$ is said to be a $PP$ - ring if the annihilator of any element of $R$ is generated by an idempotent element. Then $R$ is a $PP$ - ring if and only if $Id(R)$ is a $PP$ - quantale.

Throughout this section we shall assume that $A$ is coherent quantale.

\begin{lema}
Any $PP$ - quantale $A$ is semiprime.
\end{lema}

\begin{proof}
Firstly we remark that for any $a\in A$ such that $a\leq a^{\perp}$ and $a^{\perp}\in B(A)$ we have $a = a\land a^{\perp} = aa^{\perp} = 0$.

In order to prove that $A$ is semiprime let $c$ be a compact element such that $c^n = 0$ for some positive integer $n$. Thus $c^{n-1}\leq (c^{n-1})^{\perp}$ and $(c^{n-1})^{\perp}\in B(A)$, hence $c^{n-1} = 0$. By using many times this argument one gets $c = 0$. By using Lemma 2.4 it follows that $A$ is semiprime.

\end{proof}

\begin{lema} If $A$ is semiprime then the following hold:
\usecounter{nr}
\begin{list}{(\arabic{nr})}{\usecounter{nr}}
\item If $e\in B(A)$ then $\rho(e) = e$;

\item If $a\in A$ and $\rho(a)\in B(A)$ then $a = \rho(a)$.
\end{list}
\end{lema}

\begin{proof}
(1) Assume that $c\in K(A)$ and $c^2\leq e$. By Lemma 3.8,(1) we have $c^{\perp}\in B(A)$, hence $\lambda_A(ce^{\perp})$ = $\lambda_A(c)\land \lambda_A(e^{\perp})$ = $\lambda_A(c^2)\land \lambda_A(e^{\perp})$ =  $\lambda_A(c^2e^{\perp})$ = $\lambda_A(0) = 0$. Since $A$ is semiprime one gets $ce^{\perp} = 0$, so $c\leq e^{\perp \perp} = e$. By using the same argument, one can prove by induction that for any $c\in K(A)$ and for any positive integer $n$, $c^n\leq e$. In accordance with Lemma 2.4 one obtains $\rho(e)\leq e$, so $\rho(e) = e$.

(2) Assume that $a\in A$ and $\rho(a)\in B(A)$. By Lemma 3.9, $c = \rho(a)$ is a compact element of $A$, hence there exists a positive integer $n$ such that $c^n\leq a$. Since $c\in B(A)$ we have $c^n = c$, hence $c\leq a$. Thus $\rho(a)\leq a$, so we obtain $\rho(a) = a$.

\end{proof}

The following theorem is a generalization of a result proved by Simmons for the case of $PP$ - rings (see \cite{Simmons}).

\begin{teorema}
For a quantale $A$ let us consider the following properties:
\usecounter{nr}
\begin{list}{(\arabic{nr})}{\usecounter{nr}}

\item $A$ is a $PP$ - quantale;

\item $R(A)$ is a $PP$ - frame;

\item The reticulation $L(A)$ is a Stone lattice.

\end{list}
Then $(1)\Rightarrow(3)$ and $(1)\Leftrightarrow(2)$ hold. If $A$ is semiprime then the implication $(3)\Rightarrow(1)$ is valid.
\end{teorema}

\begin{proof}

$(1)\Rightarrow(3)$
Assume that $A$ is $PP$ - quantale. By Lemma 8.1, A is semiprime and, by Corollary 3.11, the function $\lambda_A|_{B(A)}: B(A) \rightarrow B(L(A))$ is a Boolean isomorphism. Let $x\in L(A)$, hence there exists $c\in K(A)$ such that $x = \lambda_A(c)$. Since $A$ is a $PP$ - quantale we have $c^{\perp}\in B(A)$, hence $\lambda_A(c^{\perp})\in B(L(A))$. We shall prove that $Ann(\lambda_A(c)) = (\lambda_A(c^{\perp})]$. By Proposition 3.13 we have $Ann(\lambda_A(c)) = Ann(c^{\ast}) = (c^{\perp})^{\ast}$. Let us consider an element $y\in Ann(\lambda_A(c)) = (c^{\perp})^{\ast}$, so there exists $d\in K(A)$ such that $d\leq c^{\perp}$ and $y = \lambda_A(d)$. Thus $y = \lambda_A(d)\leq \lambda_A(c^{\perp})$, so $y\in (\lambda_A(c^{\perp})]$. We have proven that $Ann(\lambda_A(c))\subseteq (\lambda_A(c^{\perp})]$.

On the other hand, from $\lambda_A(c^{\perp})\land \lambda_A(c)$ = $\lambda_A(cc^{\perp})$ = $\lambda_A(0) = 0$ we obtain $ \lambda_A(c^{\perp})\in Ann(\lambda_A(c))$, hence $(\lambda_A(c^{\perp})]\subseteq Ann(\lambda_A(c))$.

$(1)\Leftrightarrow(2)$
In accordance with Proposition 3.7, the frames $R(A)$ and $Id(L(A))$ are isomorphic.Then $L(A)$ is a Stone lattice iff $Id(L(A))$ is a $PP$ - frame iff $R(A)$ is a $PP$ - frame.

$(3)\Rightarrow(1)$
Assume now that $A$ is semiprime and $L(A)$ is a Stone lattice. Let $c$ be a compact element of $A$, so there exists $f\in B(L(A))$ such that $Ann(\lambda_A(c)) = (f]$. By Corollary 3.11 there exists $e\in B(A)$ such that $f = \lambda_A(e)$. According to Proposition 3.13 we have $(c^{\perp})^{\ast} = Ann(\lambda_A(c)) = ( \lambda_A(e)] = e^{\ast}$, hence, by using Proposition 3.3,(7) one gets $\rho(c^{\perp}) = ((c^{\perp})^{\ast})_{\ast} = (e^{\ast})_{\ast} = \rho(e)$. By Lemma 8.2,(1) we have $\rho(0) = 0$. Since $\rho(c^{\perp}) = e\in B(A)$, by applying Lemma 8.2,(2) one obtains $c^{\perp} = e$, hence $c^{\perp}\in B(A)$, so $A$ is a $PP$ - quantale.

\end{proof}

By using the previous result one can obtain characterization theorems for $PP$ - quantales by 
transferring from lattices to rings the properties that describe the Stone lattices.

\begin{lema}\cite{Cornish}
Let  $L$ be a bounded distributive lattice. Then $L$ is a Stone lattice if and only if $L$ is conormal and $Min_{Id,Z}(L)$ is compact.
\end{lema}

\begin{teorema}
For a semiprime quantale $A$ the following properties are equivalent:
\usecounter{nr}
\begin{list}{(\arabic{nr})}{\usecounter{nr}}

\item $A$ is a $PP$ - quantale;

\item The reticulation $L(A)$ is a Stone lattice;

\item $L(A)$ is a conormal lattice and $Min_{Id,Z}(L(A))$ is compact;

\item $A$ is a $PF$ - quantale and  $Min_Z(A)$ is compact;

\item $A$ is a $PF$ - quantale and  $Min_Z(A)$ is a Boolean space;

\item $A$ is an $mp$ - quantale and  $Min_Z(A)$ is a Boolean space;

\item $Spec_F(A)$ is a normal space and  $Min_Z(A)$ is a Boolean space;

\end{list}

\end{teorema}

\begin{proof}

$(1)\Leftrightarrow (2)$ By Theorem 8.3.

$(2)\Leftrightarrow (3)$ By Lemma 8.4.

$(3)\Leftrightarrow (4)$ Since $A$ is semiprime, the reticulation $L(A)$ is a conormal lattice if and only if $A$ is a $PF$ - quantale (cf. Lemma 6.1 and Theorem 6.5). On the other hand, $Min_Z(A)$ and $Min_{Id,Z}(L(A))$ are homeomorpic topological spaces. Then the equivalence of $(3)$ and $(4)$ follows.

$(4)\Leftrightarrow (5)$ By Theorem 6.2.

$(5)\Leftrightarrow (6)$ By Theorem 6.5.

$(6)\Leftrightarrow (7)$ By Theorem 6.3.

\end{proof}

\begin{lema}\cite{Cornish1}
Let $L$ be a conormal lattice. Then $L$ is a Stone lattice if and only if the inclusion $Min_Z(L)\subseteq Spec_Z(A)$ has a continuous retraction.

\end{lema}

\begin{teorema}
For a $PF$ - quantale $A$ the following properties are equivalent:
\usecounter{nr}
\begin{list}{(\arabic{nr})}{\usecounter{nr}}

\item $A$ is a $PP$ - quantale;

\item The reticulation $L(A)$ is a Stone lattice;

\item The inclusion $Min_{Id,Z}(L(A))\subseteq Spec_{Id,Z}(L(A))$ has a continuous retraction;

\item The inclusion $Min_Z(A)\subseteq Spec_Z(A)$ has a continuous retraction;

\item For any $c\in K(A)$, $Min(A)\bigcap D(c)$ is an open subset of $Spec_Z(Vir(A))$;

\item $Min_Z(A)$ is a compact space.

\end{list}

\end{teorema}

\begin{proof}

$(1)\Leftrightarrow (2)$ By Theorem 8.3.

$(2)\Leftrightarrow (3)$ By Lemma 8.6.

$(3)\Leftrightarrow (4)$ This equivalence follows because $Spec_Z(A)$ (resp. $Min_Z(A)$) is homeomorphic to $Spec_{Id,Z}(L(A))$ (resp. $Min_{Id,Z}(L(A))$) .

$(1)\Leftrightarrow (5)$ Let $c$ be a compact element of $A$, hence $c^{\perp}\in B(A)$. By Theorems 6.12 and 6.16 we have $Min(A) = Spec(Vir(A))$. According to Corolarry 3.16, for each $p\in Min(A)$ the following equivalence holds: $c\leq p$ iff $c^{\perp}\not\leq p$. Therefore one obtain the equality $Min(A)\bigcap D(c)$ = $\{p\in Spec(Vir(A))|c^{\perp}\not\leq p\}$, hence $Min(A)\bigcap D(c)$ is open in $Spec(Vir(A))$.

$(5)\Leftrightarrow (6)$ According to Lemma 4.2,(4) one can consider the composition $Vir\circ i$ of the following two continuous maps: the inclusion $i:Min_Z(A)\rightarrow Spec_Z(A)$ and $Vir:Spec_Z(A)\rightarrow Spec_Z(Vir(A))$. By Theorems 6.12 and 6.16, $Vir\circ i$ is a continuous bijection. The hypothesis (5) implies that $Vir\circ i$ is an open map, so it is a homeomorphism. Since $Spec_Z(Vir(A))$ is a compact space, it follows that $Min_Z(A)$ is also compact.

$(6)\Leftrightarrow (1)$ By Theorem 8.5,(4).

\end{proof}

\begin{corolar}
Any $PP$ - quantale is a purified quantale.

\end{corolar}

\begin{proof}
Let $A$ be a $PP$ - quantale. By Theorem 8.5,(5), $Min_Z(A)$ is a Boolean space, hence $Min_Z(A) = Min_F(A)$ (cf. Theorem 6.2,(1)). Then $Min_F(A)$ is a Boolean space, therefore, according to Theorem 7.5,(4), it follows that $A$ is a purified quantale.

\end{proof}

\end{document}